\documentclass[10pt,reqno]{amsart}
\title{Simultaneous $p$-orderings and minimising volumes in number fields}
\usepackage{amssymb,amsmath,mathrsfs,amsthm}
\usepackage{wrapfig}
\usepackage{graphicx}
\def\F{\mathbf{F}}

\def\text#1{\hbox{#1}}
\def\Q{{\mathbf Q}}
\def\Z{{\mathbf Z}}
\def\R{{\mathbf R}}

\def\C{{\mathbf C}}
\def\P{\mathbf P}

\def\Spec{{\rm Spec}}

\def\conv{{\rm conv}}
\def\col{{\rm col}}

\def\Bb#1#2{{\def\md{\bigm| }#1\bigl[\,#2\,\bigr]}}

\def\Pb{\Bb\P}

\newtheorem{theorem}{Theorem}[section]
\newtheorem{definition}{Definition}[section]
\newtheorem{lemma}[theorem]{Lemma}
\newtheorem{corollary}[theorem]{Corollary}
\newtheorem{proposition}[theorem]{Proposition}

\newtheorem{example}{Example}[section]
\newtheorem{remark}{Remark}
\newtheorem*{question*}{Question}

\newtheorem{question}[theorem]{Question}

\theoremstyle{definition}
\numberwithin{equation}{section}

\begin{document}

\author[J.~Byszewski]{Jakub Byszewski}
\address[J.~Byszewski]{Department of Mathematics and Computer Science, Institute of Mathematics, Jagiellonian University,
{\L}ojasiewicza~6, 30-348 Krak\'ow, Poland}
\email{jakub.byszewski@uj.edu.pl}

\author[M.~Fr\c{a}czyk]{Miko\l aj Fr\c{a}czyk}
\address[M.~Fr\c{a}czyk]{D\'epartement de Math\'ematiques, Universit\'e Paris-Sud, Universit\'e Paris Saclay, 91405 Orsay ´
cedex, France}
\email{mikolaj.fraczyk@math.u-psud.fr}

\author[A.~Szumowicz]{Anna Szumowicz}
\address[A.~Szumowicz]{Department of Mathematics and Computer Science, Institute of Mathematics, Jagiellonian University,
{\L}ojasiewicza~6, 30-348 Krak\'ow, Poland
\newline Institut de Math\'ematiques de Jussieu-Paris Rive Gauche,  4, Place Jussieu, 75252 Paris Cedex 05, France}
\email{anna.szumowicz@imj-prg.fr}

\keywords{Simultaneous orderings, integer-valued polynomials, generalized factorials, Euler-Kronecker constants, number fields.}

\maketitle
\begin{abstract}
In \cite{PV}, V.V. Volkov and F.V. Petrov consider the problem of existence of the so-called $n$-universal sets (related to simultaneous $\mathfrak{p}$-orderings of Bhargava) in the ring of Gaussian integers. A related problem concerning Newton sequences was considered by D. Adam and P.-J. Cahen in \cite{AC}. We extend their results to arbitrary imaginary quadratic number fields and prove an existence theorem that provides a strong counterexample to a conjecture of Volkov-Petrov on minimal cardinality of $n$-universal sets. Along the way, we discover a link with Euler-Kronecker constants and prove a lower bound on Euler-Kronecker constants which is of the same order of magnitude as the one obtained by Ihara.

\end{abstract}

\section{Introduction}
In \cite{B1} (see also \cite{B2}), M. Bhargava introduced a generalized notion of the factorial function defined in any Dedekind domain via a notion of a $\mathfrak{p}$-ordering, introduced by him as well. Let us recall here the definition of a $\mathfrak{p}$-ordering and the generalized factorial function\footnote{Bhargava defined it with respect to an arbitrary subset $S$ of the ring $A$, we recall here only a special case of the definition where $S=A$.}.
\begin{definition}
Let $A$ be a Dedekind domain and $\mathfrak{p}$ a prime ideal of $A$. Let $v_{\mathfrak{p}}$ denote the additive $\mathfrak{p}$-adic valuation on $A$. A sequence $s_0,s_1,\ldots$ of elements of $A$ (finite or infinite) is called a \textbf{$\mathfrak{p}$-ordering} if for every $n$ the element $s_n$ is chosen so that the valuation $v_{\mathfrak{p}}(\prod_{i=0}^{n-1}(s_i-s_n))$ is the lowest possible. 
We define the function $w_{\mathfrak p}(n)$ by $$w_{\mathfrak p}(n)=v_{\mathfrak{p}}\left(\prod_{i=0}^{n-1}(s_i-s_n)\right).$$ It can be shown (see \cite{B1}) that this value does not depend on the choice of a $\mathfrak{p}$-ordering. We define the \textbf{generalized factorial} of a positive integer $n$ in $A$ as the ideal
$$n!_A=\prod_{\mathfrak{p}\in \Spec \, A}\mathfrak{p}^{w_{\mathfrak p}(n)}.$$
\end{definition}
In the case $A=\Z$, we obtain the usual factorial function. However, the usual definition of the factorial function in the case $A=\Z$ is simpler, because there exists a sequence $0,1,2,3,\ldots$, which is a simultaneous $p$-ordering for all primes $p$. One can ask for which number fields $K$ there exists a simultaneous $p$-ordering in the ring of integers $\mathcal{O}_K$. This is a particular case of the question \cite[Question 30]{B2}, where Bhargava asked for which subsets $S$ of Dedekind rings there exist simultaneous $\mathfrak{p}$-orderings in $S$. For $K\neq\Q$ no example of such a sequence is known, and they are expected not to exist, but the evidence is scant. In \cite{W}, M. Wood proved that there are no simultaneous $\mathfrak{p}$-orderings in imaginary quadratic number fields. Simultaneous $\mathfrak{p}$-orderings are also called Newton sequences and have been recently studied by D. Adam and P.-J. Cahen (\cite{AC}, \cite{C}).

The problem of existence of simultaneous $\mathfrak{p}$-orderings is related to integer-valued polynomials. Let $A$ be a domain with field of fractions $K$. We call a polynomial $P\in K[X]$ {\bf integer-valued} if $f(A)\subset A$. In \cite{PV}, V.V. Volkov and F.V. Petrov define $n$-universal sets as follows.

\begin{definition} Let $A$ be a domain and $K$ its field of fractions. We call a finite subset $S\subset A$ an \textbf{$n$-universal set} if the following statement holds: For every polynomial $P\in K[X]$ of degree at most $n$ we have $P(A)\subset A$ if and only if $P(S) \subset A$.
\end{definition}

Thus, a set is $n$-universal if we can test whether a polynomial of degree $\leq n$ is integer-valued on the elements of this set only. Related to universal sets are Newton sequences.

\begin{definition} Let $A$ be a domain. A sequence $s_0,s_1,\ldots, s_n$ is a Newton sequence if for every $0\leq m\leq n$ the set $\{s_0,s_1,\ldots,s_m\}$ is $m$-universal. The integer $n$ is called the length of the Newton sequence.
\end{definition}

The minimal cardinality of an $n$-universal set is $n+1$. More precisely, we have the following lemma.

\begin{lemma}
\label{dolne}
Let $A$ be a domain which is not a field. Then every $n$-universal subset of $A$ has at least $n+1$ elements.
\end{lemma}

Indeed, by Lagrange approximation it is easy to construct a polynomial of degree $n$ that vanishes on a given set $S$ of size $n$ and takes an arbitrary value on a given element $x\not\in S$.
 
Of particular interest are thus $n$-universal sets with precisely $n+1$ elements. We call such sets \textbf{optimal $n$-universal} sets or simply \textbf{$n$-optimal} sets.

\begin{example}
The set $\{0,1,\ldots, n\}$ is an $n$-optimal set in $\Z$. More generally, $n$-optimal sets in $\Z$ have the form $\{a,a+1,\ldots,a+n\}$ for some $a\in \Z$. 
\end{example}

Indeed, the latter statement follows easily from Proposition \ref{p.NunivVol}, since the volume of a subset of $\Z$ of a given size is minimized precisely on sets of consecutive integers.

Given a ring $A$, it is natural to ask what is the size of a minimal $n$-universal set in $A$ and in particular whether $n$-optimal sets exist. For small $n$, one can construct such sets in other rings than $\mathbf{Z}$ as well, and it is very interesting to know for which rings such sets exist for all $n$. In \cite{PV}, Volkov and Petrov showed that for $K=\Q(i)$ there are no $n$-optimal sets in $\mathcal{O}_K$ for large enough $n$ and also remarked that $n$-optimal sets exist for $n=1,2,3,5$ but not for $n=4$. Furthermore, they constructed examples of $n$-universal sets of size $\frac{\pi}{2}n+o(n)$ and asked if their examples are asymptotically minimal, i.e., if the size of a minimal $n$-universal set in $\Z[i]$ grows as $\frac{\pi}{2}n+o(n)$ (see \cite[Conjecture]{PV}). We modify the geometric argument used in \cite{PV} to prove the following result.
\begin{theorem}\label{cebula}
Let $K$ be an imaginary quadratic number field. Then for large enough $n$ there are no $n$-optimal subsets of $\mathcal{O}_K$. 
\end{theorem}

This result can be considered to be a strongly negative answer to the question of Bhargava on simultaneous $\mathfrak{p}$-orderings. In fact, if for a given number field there was a simultaneous $p$-ordering $a_0,a_1,a_2,\ldots$, then the sets $A_n=\{a_0,a_1,\ldots,a_n\}$ would be $n$-optimal sets for all $n\geq 0$. Thus, our result generalizes the results of Wood \cite{W}. Quite recently, Adam and Cahen have obtained a result on the existence of Newton sequences in quadratic number fields (\cite[Theorem 16 and Section 4.1]{AC}). They prove that the length of a maximal Newton sequence in a quadratic number field is bounded except for at most finitely many exceptions (possibly none) and all of the possible exceptions are real quadratic number fields. Their result is more precise, for example they say that the maximal length of a Newton sequence in $K=\Q(\sqrt{d})$ is one if $d \not \equiv 1 \pmod{8}$ and $d\not = -3,-1,2,3,5$. However, for any natural number $m$ there are $d$ (necessarily $d\equiv 1 \pmod{8}$) such that the maximal length of a Newton sequence in $K=\Q(\sqrt{d})$ is bigger than $m$. Note that Theorem \ref{cebula} is stronger than the result of Adam and Cahen, in a sense that an $n$-optimal subset of $\mathcal{O}_K$ cannot necessarily be ordered to form a Newton sequence. For example, a Newton sequence in $K=\Q(i)$ of maximal length is $0,1,i,1+i$ and it has length three (cf. \cite[Section 4.1]{AC}). However, we have already mentioned that there is a $5$-optimal subset of $\mathcal{O}_{\Q(i)}$, namely $\{0,1,2,i,1+i,2+i\}$. This subset cannot be ordered to form a Newton sequence since it does not contain a $4$-optimal subset of $\mathcal{O}_{\Q(i)}$ (in fact, no $4$-optimal set exists!).

Since in imaginary quadratic number fields no $n$-optimal sets exist for large $n$, it is interesting to ask what is the smallest cardinality of an $n$-universal set. We show that the conjectural asymptotic lower bound $\frac{\pi}{2}n+o(n)$ proposed by Volkov-Petrov in the case of $\mathcal O_{\Q(i)}$ can be replaced by $n+2$. Our construction works in an arbitrary Dedekind domain.
\begin{theorem}
Let $A$ be a Dedekind domain. Then for any $n$ there exists an $n$-universal set in $A$ of size $n+2$.
\end{theorem}

To prove this statement, we construct an ascending sequence of $n$-universal sets with $n+2$ elements. The argument is elementary and boils down to a repeated use of Chinese Remainder Theorem, prime decomposition of ideals, and  Proposition \ref{p.CritNormal}. In fact, we have much freedom in our construction, and the constraints on $n$-universal sets with $n+2$ elements are much weaker that those on optimal $n$-universal sets.  If we were to consider $n$-universal sets with $n+d$ elements, where $d$ is the degree of the field $K$, the property of being $n$-universal is in some sense common. For any $\varepsilon>0$ we construct $n+d$ independent random walks $X_1,\ldots, X_{n+d}$ on $\mathcal{O}_K$ and put $S_m=\{X_1^{(m)},\ldots, X_{n+d}^{(m)}\}$\footnote{The notation $X^{(m)}$ denotes the $m$-th step of a random walk $X$.}. Then we show that the probability that $S_m$ is $n$-universal is at least $1-\varepsilon$ as $m$ goes to infinity. We construct random walks $X_i$ in the following way: we fix a symmetric, finitely supported probability measure $\mu$ on $\mathcal{O}_{K}$ such that the support of $\mu$ contains a basis of $\mathcal{O}_K$ over $\mathbf{Z}$  and for each $i$ we pick a starting  points $a_i$ and set $X_i^{(0)}=a_i$. Next, for $n\geq 0$  we define $X_i^{(n+1)}$ as a random element of $\mathcal{O}_K$ satisfying $\Pb{X_i^{(n+1)}=x \mid X_i^{(n)}=y} =\mu({x-y})$ for every $x,y\in \mathcal{O}_K$.

Let us return to the question of existence of $n$-optimal sets in $\mathcal{O}_K$ for general number fields $K$. As we have mentioned before, examples of small $n$-optimal sets can be constructed, so it would be very interesting to know the answer to the following question.
\begin{question*} Let $K$ be a number field and $\mathcal{O}_K$ be its ring of integers. Do there exist $n$-optimal subsets of $\mathcal{O}_K$ for arbitrarily large $n$? 
\end{question*}
Our results  give a negative answer to this question for imaginary quadratic number fields. We suspect that this is also the case in any number field $K\neq\Q$, but the presence of infinitely many units in $\mathcal{O}_K$ prevents us from extending our geometric methods.  As we explain in Section \ref{s.Volume}, the property of $n$-optimality in the rings of integers in number fields can be tested by looking at the \textbf{volume} of the set. The notion of volume and its application in the study of $n$-universal sets were already present in \cite{PV}. We modify their definition slightly and define the volume of a finite subset $S\subset \mathcal{O}_K$ as the principal ideal 
$$Vol(S)=\prod_{\substack{s_1, s_2\in S\\s_1\neq s_2}} (s_1-s_2).$$
We show that $S$ is $n$-optimal if and only if its volume is the smallest possible, i.e., divides the volumes of all sets of the same cardinality. In Section \ref{s.EKconstant}, Corollary \ref{c.FactorialEst}, we compute the asymptotic formula for the norm of the volume of an $n$-optimal set in $\mathcal{O}_K$. To this effect, we use a recent result due to M. Lamoureux. It turns out that for an  $n$-optimal set $S$ we have
$$\log N(Vol(S))={n^2}\log n-\frac{n^2}{2}-{n^2}(1+\gamma_K-\gamma_\Q)+o(n^2).$$
The constant $\gamma_{\Q}$ is the Euler-Mascheroni constant, and $\gamma_K$ is the Euler-Kronecker constant of the number field $K$. The analytic properties of $\gamma_K$ were thoroughly studied by Y. Ihara in \cite{Ihara1}. We use our estimates to prove an analytic inequality for "potential-like" integrals (Theorem \ref{t.LogIneq}) and deduce, in a elementary fashion, a lower bound on $\gamma_K$.

\subsection*{Outline of the paper:}

$ $ \medskip

 In \textbf{Section 2}, we study necessary and sufficient conditions for a subset of a ring $A$ to be $n$-universal. Following the lines of \cite{PV}, we prove that a set $S$ with $n+1$ elements in a discrete valuation ring is $n$-optimal if and only if it is almost uniformly distributed modulo powers of the maximal ideal. Next we show that in discrete valuation rings the set is $n$-universal if and only if it contains an $n$-optimal subset (Proposition \ref{p.CritDVR}). Using a local-global principle (Proposition \ref{p.CritLocal}), we prove Proposition \ref{p.CritNormal}, which gives an equivalent condition for a subset of a normal noetherian ring to be $n$-universal. In the second part of this section, we focus on  $n$-optimal sets in the rings of integers in number fields. We introduce the notion of volume of a finite subset of $\mathcal{O}_K$, which is similar to the notion of volume from \cite{PV}. We prove that being an $n$-optimal set is controlled by the volume. All these results  are collected in  Proposition \ref{p.NunivVol}, which gives equivalent conditions for a subset of $\mathcal{O}_{K}$ to be $n$-optimal.
 
 In \textbf{Section 3}, we show that for an imaginary quadratic number field $K$ there are no $n$-optimal sets in $\mathcal{O}_K$, provided that $n$ is large enough. This is an extension of the result from \cite{PV}, where such a statement was proved for $K=\Q(i)$. Put $K=\Q(\sqrt{-d})$. The proof is divided in two cases, depending on whether $d\equiv 1,2 \pmod{4}$ or $d\equiv 3 \pmod{4}$. The case  $d\equiv 1,2 \pmod{4}$ is similar to the proof of Volkov and Petrov, but when $d\equiv 3 \pmod{4}$ the ring of integers $\mathcal{O}_K$ viewed as a lattice in $\mathbf{C}$ has a slightly different geometry and the argument is more involved.

 \textbf{Section 4} is entirely devoted to the construction of $n$-universal sets with $n+2$ elements. We also discuss the computational complexity of this construction.
 
 In \textbf{Section 5}, we compute the asymptotic volume of $n$-optimal subsets of $\mathcal{O}_K$, where $K$ is a number field (assuming they exist). The Euler-Kronecker constant appears naturally in those formulae. We use our estimates to prove an analytic inequality (Theorem \ref{t.LogIneq}), which is similar to those appearing in  potential theory. We use the latter to obtain an elementary proof of a lower bound on the Euler-Kronecker constant $\gamma_K$ of strength comparable to the lower bound of Ihara \cite{Ihara1}. It is interesting to note that the result is independent of existence of $n$-optimal sets. 
 
In \textbf{ Section 6}, we construct a random walk on $\mathcal{O}_K^{n+d}$ where $d=[K:\Q]$ such that the probability that the coordinates of an $m$-th step form an $n$-universal set tends to $1$ as $m$ tends to infinity. The methods used involve harmonic analysis on finite groups.

\subsection*{Notations:}

$ $ \medskip

By $|S|$, we denote the cardinality of a set $S$. Let $f_{1}, f_{2}$ be real valued functions defined on a subset of natural numbers. We write $f_{2}(n)=o(f_{1}(n))$ if $\lim_{n\to \infty}\frac{f_2(n)}{f_1(n)}=0$ and $f_{2}(n)=\Omega (f_{1}(n))$ if $\mathrm{liminf}_{n\to\infty}\frac{f_2(n)}{f_1(n)}>0$. More generally, whenever we have two expressions $f$ and $g$ (that might depend on some parameters), we shall write $f\ll g$ if there exists a constant $C>0$ such that $f\leq Cg$ for large values of the parameters. Throughout the paper we will write $\mathcal{O}_K$ for the ring of integers in a number field $K$, $\Delta_K$ for the discriminant of $K$ and $A_{\frak p}$ for the localisation of a ring $A$ in a prime ideal $\frak p$. We denote by $N(I)$ the norm of an ideal $I$ in a ring $A$ defined as $N(I)=|A/I|$. If $K$ is a number field and $\frak p$ is a prime ideal of $\mathcal{O}_K$, then we write $K_\frak p$ for the completion of $K$ with respect to the $\frak p$-adic topology. All logarithms appearing throughout the article are taken with respect to the natural base.

\section{Universal and optimal sets}  This section is inspired by the results obtained by Volkov and Petrov in \cite{PV} on the properties of $n$-universal sets. They found equivalent conditions for a subset of a unique factorization domain to be $n$-optimal (and one of these conditions holds also in the more general case of integral domains). In this section, we shall focus on criteria for $n$-universality of sets of arbitrary size. 
In our considerations, we restrict ourselves to normal noetherian domains. Our aim is Proposition \ref{p.CritNormal} -- a practical criterion for a subset of a normal noetherian domain to be $n$-universal and Proposition \ref{p.NunivVol} which gives equivalent conditions for $n$-optimality in terms of volume. 
Some of the results in this section are likely part of the folklore, but we include them here for the convenience of the reader and due to lack of suitable reference. Some arguments here generalize the results of \cite{PV}. Universal sets are closely connected with integer-valued polynomials. 
For a comprehensive discussion of integer-valued polynomials, see \cite{CC}. Almost uniformly distributed sequences have been studied in particular by Amice, Bhargava, and Yeremian (\cite{A}, \cite{B1}, \cite{B2} and \cite{Y}).

\begin{definition}\label{d.AUD} Let $A$ be an integral domain, $I$ an ideal of $A$ and $S$ a finite subset of $A$. We say that $S$ is \textit{almost uniformly distributed} modulo $I$ if for every $a,b\in A$ we have
$$|\{s\in S \mid s\equiv a \pmod I\}|-|\{s\in S \mid s\equiv b \pmod I\}|\in \{-1,0,1\}.$$
In particular, if $|A/I|\geq |S|$, then $S$ is almost uniformly distributed modulo $I$ if and only if its elements are pairwise distinct modulo $I$.
\end{definition}

In the following $S$ will denote a subset of $A$. 
\begin{definition}
Let $A$ be an integral domain and let $\mathfrak{p}$ be a prime ideal of $A$. A set $S$ is called $(n,\mathfrak{p})$-universal if $S$ is $n$-universal in the localisation $A_{\mathfrak{p}}$. If this is the case, $f(S)\subset A$ implies $f(A)\subset A_{\mathfrak{p}}$.
\end{definition}
\begin{proposition}\label{p.CritDVR}
Let $A$ be a discrete valuation ring with the maximal ideal $\mathfrak{m}$. Then a subset $S\subset A$ is $n$-universal if and only if it contains a subset of size $n+1$ which is almost uniformly distributed modulo $\mathfrak{m}^k$ for every positive integer $k$.
\end{proposition}
\begin{proof}
\textbf{Step 1:} First we treat the case $|S|=n+1$. This is already handled by Lemma $1$ from \cite{PV}, we just need to observe that every discrete valuation ring is a unique factorization domain. Lemma $1$ yields then that $S$ is $n$-universal if and only if it is almost uniformly distributed modulo $\mathfrak{m}^k$ for every positive integer $k$. Let us recall the first part of the proof from \cite{PV}. We shall refer to it in the proof of the second step. Let $S=\{s_0,s_1,\ldots,s_n\}$. For $m=0,1,\ldots, n$ consider the polynomial 
$$Q_m(X)=\prod_{i\neq m}\frac{X-s_i}{s_m-s_i}.$$
Then $Q(s_i)=0$ if $i\neq m$ and $Q(s_m)=1$. Let $f\in K[X]$ be a polynomial of degree at most $n$. By Lagrange interpolation, we get
$$f(X)=\sum_{m=0}^nf(s_m)Q_m(X).$$
This formula implies that if all the polynomials $Q_m$ are integer-valued, then $S$ must be $n$-universal. Conversely, if $S$ is $n$-universal, then the polynomials $Q_m$ must  be integer-valued, so the two conditions are in fact equivalent. Thus,  it is enough to show that polynomials $Q_m$ are integer-valued if and only if $S$ is almost uniformly distributed modulo $\mathfrak{m}^k$ for every positive integer $k$. Let $v$ be the (additive) valuation of $A$. Rewrite the condition that $Q_m$ is integer-valued in the form
$$\sum_{i\neq m}v(x-s_i)\geq \sum_{i\neq m}v(s_m-s_i) \text{ for } x\in A.$$
The second part of the proof of Lemma 1 in \cite{PV} yields that such inequality holds for all $x$ and $m$ if and only if $S$ is almost uniformly distributed modulo $\mathfrak{m}^k$ for every positive integer $k$.

\textbf{Step 2:} Let $S$ be an $n$-universal set and $|S|>n+1$. Take a set $E$ for which the sum $\sum_{s,s'\in E, s\neq s'}v(s-s')$ is minimal among all subsets of $S$ with $n+1$ elements. If $E$ is not $n$-universal then, by the proof of the first step, it means that for each such $E$ there exists (at least one) $s_E\in E$ such that
$$Q_{E,s_E}(X)=\prod_{s\in E\setminus \{s_E\}}\frac{X-s}{s_E-s}$$
is not integer-valued. Since $S$ is supposed to be $n$-universal, this means that there exists $r_E\in S$ such that $Q_{E,s_E}(r_E)\not\in A,$ i.e,
$$\sum_{s\in E\setminus \{s_E\}}v(r_E-s)<\sum_{s\in E\setminus \{s_E\}}v(s_E-s).$$
Consider the set $E'=E\setminus\{s_E\}\cup\{r_E\}$. We get
\begin{align*}
\sum_{\substack{s,s'\in E'\\ s\neq s'}}v(s-s')&=\sum_{\substack{s,s'\in E \setminus \{s_E \}\\ s\neq s'}}v(s-s')+2\!\!\!\!\sum_{s\in E\setminus\{s_E\}}v(r_E-s)<\\
&< \sum_{\substack{s,s'\in E \setminus \{s_E \}\\ s\neq s'}}v(s-s')+2\!\!\!\!\sum_{s\in E\setminus\{s_E\}}v(s_E-s)=
\sum_{\substack{s,s'\in E\\ s\neq s'}}v(s-s').
\end{align*}
This contradicts the choice of $E$. Hence, $E$ is an optimal $n$-universal subset of $S$, which concludes the proof by Step 1.\end{proof}

\begin{remark}\label{r.Rem1}
In a discrete valuation ring an optimal set can be ordered to form a Newton sequence. This follows immediately from the previous proposition. 
\end{remark}

The property of being an $n$-universal set is local and in general it is easier to verify in local rings. 

\begin{proposition}\label{p.CritLocal}
Let $A$ be an integral domain with field of fractions $K$. Let $S$  be a subset of $A$ and let $X$ be a set of prime ideals of $A$ such that $$\bigcap_{\mathfrak{p}\in X}A_\mathfrak{p}=A.$$
Then $S$ is $n$-universal in $A$ if and only if it is $(n,\mathfrak{p})$-universal for every $\mathfrak{p}\in X$.
\end{proposition}
\begin{proof}
If $S$ is $n$-universal in $A$, then it is obviously $n$-universal in any localisation. Now assume that $S$ is $(n,\mathfrak{p})$-universal for all $\mathfrak{p}\in X$ and for a polynomial $P\in K[X]$ of degree at most $n$ we have $P(s)\in A_{\mathfrak{p}}$ for every $s\in S$. Then for any $a\in A$ and $\mathfrak{p}\in X$ we have $P(a)\in A_\mathfrak{p}$ and consequently $P(a)\in \bigcap_{\mathfrak{p}\in X}A_\mathfrak{p}=A$.
\end{proof}
As a corollary to Propositions \ref{p.CritDVR} and \ref{p.CritLocal}, we obtain the following result.
\begin{proposition}\label{p.CritNormal}
Let $A$ be a normal noetherian domain. Then a subset $S\subset A$ is $n$-universal if and only if for every prime ideal $\mathfrak{p}$ of codimension $1$ it contains a subset $S'\subset S$ with $n+1$ elements which is almost uniformly distributed modulo every positive power of $\mathfrak{p}$.
\end{proposition}
\begin{proof}
The equality
$$\bigcap_{{\rm codim} \mathfrak{p}=1}A_{\mathfrak{p}}=A$$
holds in every normal noetherian domain $A$ (cf.~\cite[Corollary 11.4]{Eis}). Thus, we can apply Proposition \ref{p.CritLocal}. For every prime ideal $\mathfrak{p}$ of codimension $1$ the ring $A_{\mathfrak{p}}$ is a discrete valuation ring. The claim follows from Proposition \ref{p.CritDVR}.
\end{proof}
The ring of integers in a number field $K$ is always a Dedekind domain, so in particular every non-zero prime ideal is maximal and has codimension one. We obtain the following corollary.
\begin{corollary}\label{p.n-univCondition}
Let $K$ be a number field and $S$ be a subset of $\mathcal{O}_K$. Then $S$ is $n$-universal if and only if for any non-zero prime ideal $\mathfrak{p}$ of $\mathcal{O}_K$ we can find a subset $S'$ of $S$ with $n+1$ elements which is almost uniformly distributed modulo powers of $\mathfrak{p}$.
\end{corollary}
\subsection{Optimal sets and volume.}\label{s.Volume}
For optimal sets, we can rephrase the last result in terms of the notion of a \textbf{volume} of a set.
\begin{definition}
Let $S$ be a finite subset of $\mathcal{O}_K$. Define the volume of $S$ as the ideal
$$Vol(S)=\prod_{\substack{s,s'\in S\\ s\neq s'}}(s-s').$$ 
\end{definition}
The volumes of optimal $n$-universal sets are closely related to the factorial function in number fields. We recall the definition below. These factorials can be traced back to the work of Bhargava on  $p$-orderings and generalizations of factorial function \cite{B1}. We give an equivalent description of the generalized factorial function for the rings of integers in number fields.
\begin{definition}(cf. \cite{B1}, \cite{MLam14}) Let $K$ be a number field. We define the $K$-factorial of $n$ as the ideal 
$$n!_K=n!_{\mathcal{O}_K}=\prod_{\mathfrak{p} \in Spec\,\mathcal{O}_K}\mathfrak{p}^{w_{\mathfrak p}(n)},$$
where $w_{\mathfrak p}(n)=\sum_{i=1}^{\infty}\left[ n/N(\mathfrak{p}^i)\right]$.
\end{definition}
The rate of growth of norms of these factorials has been studied in depth in the recent thesis of M. Lamoureux \cite{MLam14}. We will use one of his results in Section \ref{s.EKconstant}. The following lemma provides a link between $n$-optimal subsets of $\mathcal{O}_K$ and generalized factorials.
\begin{lemma}\label{l.locAEq}
Let $A$ be a discrete valuation ring with the additive valuation $v$, the maximal ideal $\mathfrak{p}$,  and finite residue field. Then $S\subset A$ is  $n$-optimal if and only if $|S|=n+1$ and $$v(Vol(S))=2\sum_{k=1}^n\sum_{i=1}^{\infty}\left[ k/N(\mathfrak{p}^i)\right].$$
\end{lemma}
\begin{proof} 
%By Theorem 2.2 from \cite{MWoods} and equivalence between (1) and (3) in the Propositin \ref{p.NunivVol} we have 
%$$v(Vol(S))=\sum_{k=1}^{n}v(k!_{A}),$$ where $k!_A$ denote the generalized factorial in $A$.\footnote{for definition look at \cite{Bhar} or \cite{MWoods}.} It remains to verify that $v(k!_A)=\sum_{i=1}^{\infty}\left[ \frac{k}{N(\mathfrak{p}^i)}\right]$, . 
By Proposition \ref{d.AUD}, the set $S$ is $n$-optimal in $A$ if and only if it is almost uniformly distributed modulo powers of $\frak p$. In a discrete valuation ring any such set can be arranged into a Newton sequence (cf. Remark \ref{r.Rem1}). Hence, the $\frak p$-valuation $v(Vol(S))$ equals $v(Vol(
\{a_{1},\ldots,a_{|S|}\}))$ where $\{a_{1},\ldots,a_{|S|}\}$ is a $\frak p$-ordering of size $|S|$. We get
$$v(Vol(S))=\sum_{i\neq j}v(a_i-a_j)=2\sum_{i=2}^{|S|}\sum_{j=1}^{i-1}v(a_i-a_j)=2\sum_{i=2}^{|S|}v(i!_A).$$ From the discussion in Chapter 1.1 in \cite{MLam14}, we have $v(i!_A)=\sum_{k=1}^{\infty}\left[ i/N(\mathfrak{p}^k)\right].$ 
\end{proof}
\begin{proposition}\label{p.NunivVol}
Let $S$ be a subset of $\mathcal{O}_K$ with $n+1$ elements. Then the following conditions are equivalent:
\begin{enumerate}
\item $S$ is $n$-optimal,
\item $Vol(S)=(\prod_{i=1}^n i!_K)^2$, 
\item $Vol(S)$ divides $Vol(S_1)$ for any subset $S_1\subset \mathcal{O}_K$ with $n+1$ elements.
\end{enumerate}
\end{proposition}
\begin{proof}
The equivalence between (i) and (iii) is a special case of Lemma 1 in \cite{PV}. For the sake of the %popr
reader, we sketch a short proof. In a discrete valuation ring $\mathcal{O}_{K_\mathfrak{p}}$ a set $S$ with $n+1$ elements is $n$-universal  if and only if the $\mathfrak{p}$-adic valuation of $Vol(S)$ is the lowest possible. Since a set is $n$-universal if and only if it is $n$-universal in the localisation in every prime ideal, we obtain the desired equivalence.
To prove that (i) is equivalent to (ii), we pass to the localisation at every prime ideal $\mathfrak{p}$ and use Lemma \ref{l.locAEq}. 
\end{proof}

\section{Nonexistence of optimal sets in imaginary quadratic number fields}

In the ring of rational integers we can find an $n$-optimal set for every $n$.  To the best of our knowledge, there are no known examples of number fields other than $\Q$ where this property holds. Volkov and Petrov showed in \cite{PV} that for  large enough $n$ there are no $n$-optimal sets with elements in $\mathbf{Z}[i]$. Thus, we can ask the following question.
\begin{question}
Let $K$ be a number field other than $\Q$. Can we find a number $N$ such that for every $n>N$ there is no $n$-optimal subset of $\mathcal{O}_K$?
\end{question}
As we have mentioned, the work of Volkov and Petrov answers this question affirmatively for $K=\Q(i)$. In this section, we will prove that the answer is positive for all imaginary quadratic number fields. 

\begin{theorem}\label{t.ImQuad}
\label{3}
Let $K=\mathbf{Q}(\sqrt{-d})$ be an imaginary quadratic number field. Then for every large enough $n$ there is no $n$-optimal set in $\mathcal{O}_{K}$.
\end{theorem}
\begin{proof}
We divide the proof into two cases. In the case $d\equiv 1,2 \pmod{4}$, we will use methods similar to those in \cite{PV}. In the case $d\equiv 3\pmod{4}$, the problem is more difficult and we need to introduce a number of modifications.

Assume the contrary, i.e., that for arbitrarily large $n$ there exists an $n$-optimal set. From now on we identify $\mathcal{O}_{K}$ with its image via a fixed embedding $K\hookrightarrow \mathbf{C}$. The strategy of the proof is as follows. Let $S$ be an $n$-optimal set. We show that $S$ is contained in a polygon which contains $n+o(n)$ points from the lattice $\mathcal{O}_{K}$. Then we show that since every $n$-optimal set is almost uniformly distributed modulo $\mathfrak{p}^{s}$ for every prime $\mathfrak{p}$ and $s\geq 1$, there exists a subset of lattice points of the polygon which is disjoint with $S$ and has $\Omega(n)$ points. This yields a contradiction.

It is well known that the ring of integers $\mathcal{O}_{K}$ is equal to $\mathbf{Z}[\frac{1+\sqrt{-d}}{2}]$ if $d\equiv -1 \pmod{4}$ and $\mathcal{O}_K=\mathbf{Z}[\sqrt{-d}]$ otherwise. Throughout the proof $S$ denotes an $n$-optimal set.

\subsection{Case $d\not\equiv -1 \pmod{4}$}

Take $\epsilon >0$. We use Proposition \ref{p.NunivVol} to show that an $n$-optimal set is collapsed\footnote{This notion will be made precise later.} along some horizontal and vertical lines. Together with Corollary \ref{p.n-univCondition}, this implies that the set $S$ is contained in a rectangle with sides parallel to the coordinate axes and  containing $n+o(n)$ points from $\mathcal{O}_{K}$. We now introduce the notion of collapsing.
\begin{definition}\label{d.Collapse}
Let $K$ be an imaginary quadratic number field and let $T$ be a finite subset of $\mathcal{O}_{K}$. Let $l$ be a line in the complex plane. The line $l$ divides the complex plane into two closed half-planes $H_1,H_2$. Let us distinguish one of them, say $H=H_1$. 
\newline We say that the set $T$ is \textbf{collapsed} along the pair $(l,H)$  if the following conditions holds: 
\begin{enumerate}
\item Let $m$ be a line perpendicular to $l$, containing at least one point from the set $T$. Let $x$ be a point from $T$ which belongs to $m$. Then every point in $\mathcal{O}_K$ lying between the point $m\cap l$ and $x$ also belongs to $T$.\footnote{In other words $T\cap m=\mathcal{O}_K\cap \conv ((T\cup l)\cap m)$.} 
\item Let $m$ be a line perpendicular to $l$. Then $|T\cap m\cap H_1|-|T\cap m\cap H_2|$ is equal to $0$ or $1$.
\end{enumerate}

We call a closed domain which is bounded by two lines parallel to $l$ and symmetric with respect to $l$ a \textbf{strip} along  $l$. A \textbf{semi-strip} along $l$ is a strip along $l$ without the part of the boundary lying in the distinguished half-plane. A strip (resp. semi-strip) parallel to $l$ is a strip (resp. semi-strip) along a line parallel to $l$.
\end{definition}
\begin{definition} Let $T$ be a finite subset of $\mathcal{O}_K$ and $(l,H)$ be as above. The set $\col_{(l,H)}(T)$ is defined as the unique subset of $\mathcal{O}_K$ satisfying the following properties:
\begin{enumerate}
\item $\col_{(l,H)}(T)$ is collapsed along $(l,H)$,
\item for every line $m$ perpendicular to $l$ we have $|T\cap m|=|\col_{(l,H)}(T)\cap m|$.
\end{enumerate}
\end{definition} 
We will now show that the set $S$ is collapsed along some vertical and horizontal lines. We use the following lemma.
\begin{lemma}
\label{collapsing}
Let $T$ be a finite subset of $\mathcal{O}_{K}$, where $K$ is an imaginary quadratic number field. Let $l$ be a line in the complex plane. If $T$ is an $n$-optimal set, then there exists a line $l_{1}$ parallel to $l$ such that the set $T$ is collapsed along the line $l_{1}$ (for some choice of the distinguished half-plane).
\end{lemma} 
\begin{proof}
By Proposition \ref{p.NunivVol}, the absolute value of the volume of the set $T$ is minimal among absolute values of the volumes of the subsets of $\mathcal{O}_{K}$ of the same cardinality as $T$. The absolute value of the volume of the set $T=\lbrace a_{0}, \ldots, a_{n} \rbrace$ is given by \begin{equation}
\label{VOL}
|Vol(T)|=\prod _{\substack{k,m-\textrm{lines} \\ k\parallel l, m\parallel l}} \prod_{\substack{a_{i}\in k\\a_{j} \in m\\i\neq j}}|a_{i}-a_{j}|.
\end{equation}
Due to this formula, it is enough to show the following lemma:
\begin{lemma}
\label{prosta}
Let $l_{2}$ and $l_{3}$ be two parallel lines in the complex plane and let $C,D$ be finite subsets of $\mathcal{O}_{K}$ such that $C$ is contained in $l_{2}$ and $D$ is contained in $l_{3}$. The number of elements $c\in C$ and $d\in D$ such that $|c-d|\leqslant m$ is maximal for all $m\geq 0$ if and only if there exists a line $t$ in the complex plane (perpendicular to $l_{2}$ and $l_{3}$) such that the sets $C$ and $D$ are collapsed along the line $t$ (up to a choice of the distinguished half-plane).
\end{lemma}
\begin{proof}
Let $d\in D$. Denote by $d_{1}$ the orthogonal projection of $d$ onto the line $l_{2}$ and let $D_{1}=\lbrace d_{1}  \mid d\in D \rbrace$. It is enough to prove the lemma in the case $l_{2}=l_{3}$, $C=C$ and $D=D_{1}$ because the function $|c-d|\mapsto |c-d_{1}|$ is strictly increasing. The proof of this case is analogous to the proof of Lemma 4 in \cite{PV}.
\end{proof}
We argue that if the absolute value of the volume of the set $T$ is minimal among the absolute values of volumes of  subsets of $\mathcal{O}_{K}$ of the same cardinality as $T$, then for every pair of lines $m_1,m_2$ perpendicular to $l$  there exists a line $l_{12}$ parallel to $l$ such that the set $T\cap (m_1\cup m_2)$ is collapsed along the line $l_{12}$. Indeed, if it was not the case then by Lemma \ref{prosta} and formula (\ref{VOL}) the absolute value $|Vol(\col_{(l,H)}(T))|$ would be strictly smaller than $|Vol(T)|$. It remains to show that we can choose a single line $l_1$ parallel to $l$ and a half-plane $H$  such that $T$ is collapsed along $(l,H)$.

 For every line $m$ perpendicular to $l$, consider the set $C_{m}$ of lines $l'$ parallel to $l$ such that $T\cap m$ is collapsed along $l'$ (for some choice of a half-plane). By Definition \ref{d.Collapse} the lines in the set $C_{m}$ form a semi-strip parallel to $l$. As we have pointed out, for every pair of lines $m_1,m_2$ perpendicular to $l$ there exists a line $l_{12}$ parallel to $l$ and a half-plane $H_{12}$ such that $(T\cap m_1)\cup (T\cap m_2)$ is collapsed along $(l_{12},H_{12})$. This is equivalent to saying that $C_{m_1}\cap C_{m_2}$ is nonempty for any $m_1,m_2$. The sets $C_{m}$ for $m$ perpendicular to $l$ are semi-strips parallel to $l$, hence they intersect non-trivially if and only if every pair does.  Hence $\bigcap C_{m}\neq \emptyset$ and we can find a line $k$ parallel to $l$ such that $T\cap m$ is collapsed along $k$ for every $m$. We can distinguish the same half-plane for every $m$ since we can do it for every pair $m_1,m_2$. Thus, $T$ is itself collapsed along $k$, which ends the proof of Lemma \ref{collapsing}.
\end{proof}

By the previous lemma, as the set $S$ is $n$-optimal, there exists a horizontal line $l_{1}$ and a vertical line $l_{2}$ such that the set $S$ is collapsed along $l_{1}$ and $l_{2}$. There exist strips $S_{l_{1}}$ along the line $l_{1}$ and $S_{l_{2}}$ along the line $l_{2}$ such that the set $S$ is contained in the intersection $S_{l_1}\cap S_{l_2}$. We will use Corollary \ref{p.n-univCondition} to calculate a bound for the width of these strips. To this end, we will use the following results.
\begin{theorem}
\label{prime}
Let $\epsilon >0$ and $a, b$ be coprime natural numbers. Then for every $m$ large enough, there exists a prime number $p \in (m,(1+\epsilon )m)$ such that $p=af+b$ where $f$ is an integer.
\end{theorem}
This is a standard consequence of the prime number theorem and the Dirichlet's theorem on prime numbers in arithmetic progressions (more precisely, its version with natural density, cf. \cite[Ch. VIII.4 and Ch. XV]{Lan}). A more detailed proof may be found in \cite{PV}. The following lemma is well-known, but we supply its proof for a lack of suitable reference.
\begin{lemma}
For every nonzero natural number $d$ there exists a natural number $c$ coprime with $4d$ such that all numbers of form $4dl+c$ which are prime in $\mathbf{Z}$ are irreducible in the ring of integers $\mathcal{O}_{K}$ of the number field $K=\mathbf{Q}(\sqrt{-d})$.
\end{lemma}
\begin{proof}
An odd prime number $p$ is irreducible in the ring of integers $\mathcal{O}_{K}$ if and only if $-d$ is not a square modulo $p$. Let us write $d=2^kd_1$ with $d_1$ odd,  $k \in\{0,1\}$.  The quadratic reciprocity for Jacobi symbols (see e.g. \cite{Ire}) yields $$\left(\frac{p}{d_1}\right)\left(\frac{d_1}{p}\right)=(-1)^\frac{(p-1)(d_1-1)}{4}.$$ Consequently, if $p\equiv 1\pmod 4$, then $$\left(\frac{d_1}{p}\right)=\left(\frac{p}{d_1}\right).$$ Let $c$ be a natural number such that $\left(\frac{c}{d_{1}}\right)=-1$ and $c \equiv 1 \pmod 4$. Then for every prime $p$ of the form $4dl+c$, we get
$$\left(\frac{-d}{p}\right)=\left(\frac{-1}{p}\right)\left(\frac{2}{p}\right)^k\left(\frac{d_1}{p}\right)=\left(\frac{p}{d_1}\right)=\left(\frac{c}{d_1}\right)=-1,$$ since $\left(\frac{2}{p}\right)^k=1$ (this is obvious for $k=0$; for $k=1$ this follows from the fact that in this case $p \equiv 1 \pmod 8$).  We conclude that for primes $p$ of the form $4dl+c$, the number $-d$ is not a square modulo $p$ and so $p$ is irreducible in $\mathcal{O}_K$.\end{proof} 
Let us take a natural number $c$ such that any prime number of the form $4dl+c$ is irreducible in the ring of integers $\mathbf{Z}[\sqrt{-d}]$. By Theorem \ref{prime}, there exists a prime number $4dl+c \in (m, (1+\epsilon )m)$ provided that $m\in \mathbf{N}$ is large enough. Let us take prime numbers $p_{1}\in (\sqrt{n}, (1+\epsilon)\sqrt{n})$ and $p_{2}\in (\sqrt{\frac{n}{2}}, (1+\epsilon )\sqrt{\frac{n}{2}})$ such that $p_{1}$ and $p_{2}$ are irreducible in $\mathbf{Z}[\sqrt{-d}]$. There are $p_{1}^{2}> n$ different remainders modulo $p_{1}$, so by Corollary \ref{p.n-univCondition} two different elements from the set $S$ cannot give the same remainder modulo $p_{1}$. Therefore, as the set $S$ is collapsed along the lines $l_{1}$ and $l_{2}$, we can assume that the strip $S_{l_{1}}$ has width $p_{1}\sqrt{d}$ and the strip $S_{l_{2}}$ has  width $p_{1}$. Hence, the set $S$ is contained in the rectangle $P=S_{l_1}\cap S_{l_2}$. The rectangle $P$ has horizontal and vertical sides of length respectively $p_{1}$ and $p_{1}\sqrt{d}$ and its left bottom vertex belongs to $\mathcal{O}_{K}$. The rectangle $P$ contains $(p_{1}+1)^{2}$ points from the lattice. Now we will show that a substantial subset of the lattice points in the rectangle $P$ cannot belong to the set $S$. Let $P_{1}$ be the rectangle with horizontal and vertical sides of lengths respectively $p_{1}-p_{2}-1$ and $(p_{1}-p_{2}-1)\sqrt{d}$ which is contained in the rectangle $P$ and such that the left bottom vertex of the rectangle $P_{1}$ is the left bottom vertex of the rectangle $P$. The rectangle $P_{1}$ contains $(p_{1}-p_{2})^{2}$ points from $\mathcal{O}_{K}$. Let $x$ be a lattice point which is contained in the rectangle $P_{1}$. Consider the set $R_{x}=\lbrace x, x+p_{2}\sqrt{-d}, x+p_{2}\sqrt{-d}+p_{2}, x+p_{2} \rbrace$. The set $R_{x}$ is contained in the rectangle $P$. Each element from the set $R_{x}$ has the same remainder modulo $p_{2}$. There are $p_{2}^{2}>\frac{n}{2}$ different remainders modulo $p_{2}$. By  Corollary \ref{p.n-univCondition}, the set $S$ contains at most two elements from the set $R_{x}$. Considering the set $R_{x}$ for every $x\in P_{1}\cap \mathcal{O}_{K}$, we get $$ |S|\leq (p_{1}+1)^{2}-2(p_{1}-p_{2})^{2}.$$ This implies that \begin{align*}|S|&\leq -p_{1}^{2}+4p_{1}p_{2}-2p_{2}^{2}+2p_{1}+1\\ n+1&\leq -2n+\frac{4}{\sqrt{2}}n(1+\epsilon )^{2}+2(1+\epsilon )\sqrt{n} +1 \\ \frac{3\sqrt{2}}{4}&\leq (1+\epsilon  )^{2}+\frac{(1+\epsilon )\sqrt{n}}{\sqrt{2}n}.\end{align*} As we can take $\epsilon>0$ arbitrarily small, we get a contradiction for large $n$. This ends the proof of Case 1 of Theorem \ref{3}.
\subsection{Case $d\equiv -1 \pmod{4}$.}
Recall that  $\mathcal{O}_{K}=\mathbf{Z}[\frac{1+\sqrt{-d}}{2}]$. Take $\epsilon > 0$. Denote by $k_{1}$ (resp. $ k_{2}$) the line which contains the point $0$ and is perpendicular to the line which contains points $0$ and $\frac{1+\sqrt{-d}}{2}$ (resp. points $0$ and $\frac{-1+\sqrt{-d}}{2}$). We have $$k_{1}=\lbrace (x,y)\in \mathbf{R}^{2}\mid y=- \frac{x}{\sqrt{d}}\rbrace$$ $$k_{2}=\lbrace (x,y)\in \mathbf{R}^{2}\mid y=\frac{x}{\sqrt{d}}\rbrace .$$
By Lemma \ref{collapsing}, there exist lines $m_{1},m_{2}$ parallel respectively to $k_{1}$ and $k_{2}$ and  a vertical line $m_{3}$ such that the set $S$ is collapsed along the lines $m_{1},m_{2}$, and $m_{3}$.

Denote by $m_{12}$ the intersection of $m_{1}$ and $m_{2}$ and by $x$ the distance between the point $m_{12}$ and the line $m_{3}$. We will divide the proof of this case into two subcases which depend on the distance $x$. Choose $p_{1}$ and $p_{2}$ in the same way as in the previous case, i.e., $p_{1}\in (\sqrt{n}, (1+\epsilon)\sqrt{n})$ and $p_{2}\in (\sqrt{\frac{n}{2}}, (1+\epsilon )\sqrt{\frac{n}{2}})$ such that $p_{1}$ and $p_{2}$ are irreducible in $\mathbf{Z}[\frac{(1+\sqrt{-d})}{2}]$.
Choose constants $C_1$ and $C$ such that  $\frac{3\sqrt{2}}{4(1+2\epsilon )}<C_{1}< \frac{3\sqrt{2}}{4(1+\epsilon )}$ and $0<C<\frac{p_{1}+1-C_{1}p_{2}}{2}$. (This is indeed possible for small $\epsilon$.)

\textbf{Subcase 1.} Assume $x\leqslant C$. In this case, we will show that the $n$-optimal set $S$ is contained in a hexagon which contains at most $p_{1}^{2}+o(p_{1}^{2})$ points from $\mathcal{O}_{K}$. Then we will show that some points from the hexagon cannot belong to the set $S$. The cardinality of the set of these points is $\eta n+o(n)$ (for some $\eta >0$ independent of $n$) and we get a contradiction for large $n$.

\begin{figure}[h]
  \centering
\includegraphics[width=140mm]{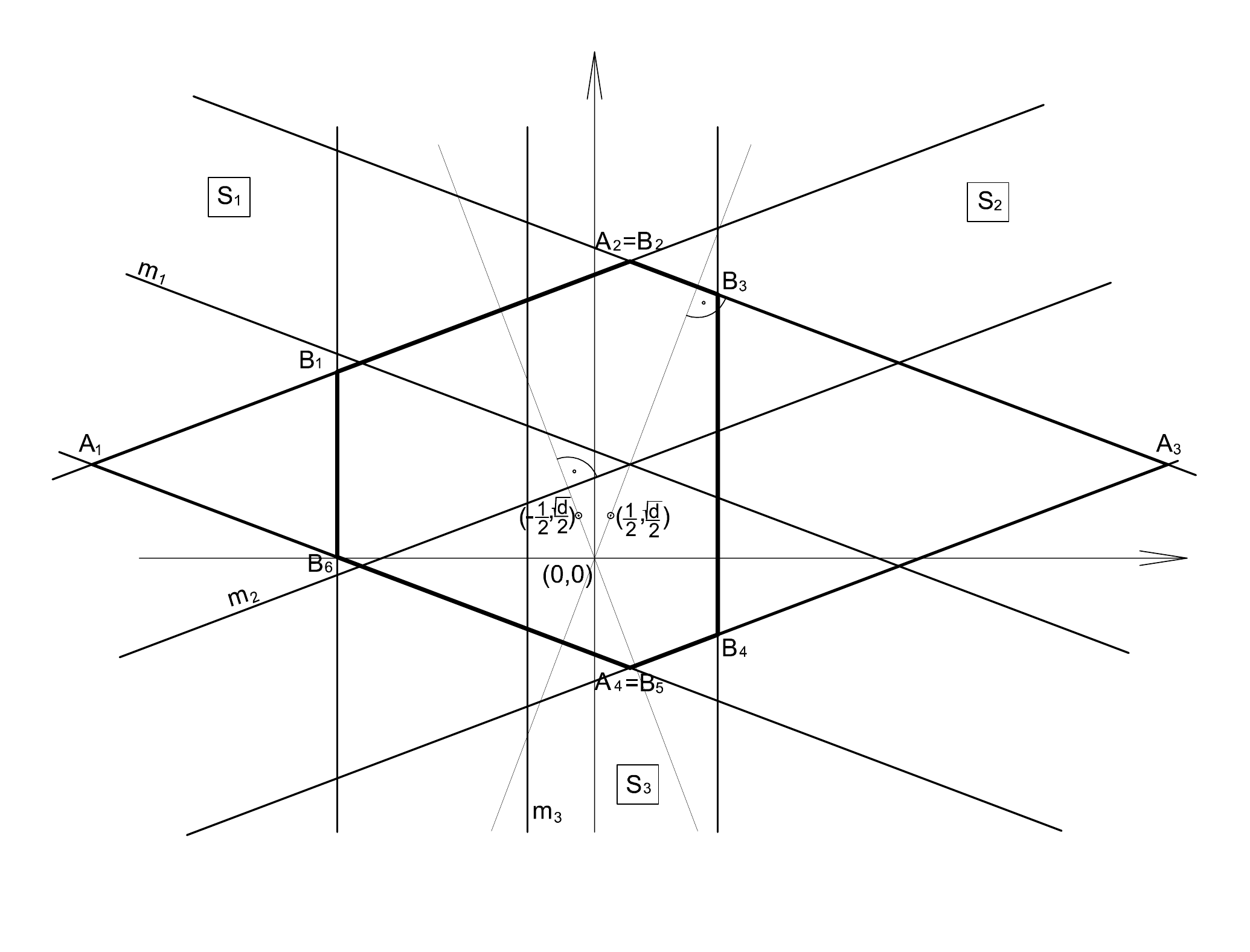}
\caption{The hexagon $B_{1}B_{2}B_{3}B_{4}B_{5}B_{6}$ contains the set $S$.}
\end{figure}

The set $S$ is collapsed along the lines $m_{1},m_{2}, m_{3}$ and is contained in the intersection of the strips: $S_{1}$ along the line $m_{1}$, $S_{2}$ along the line $m_{2}$ and $S_{3}$ along the line $m_{3}$ (see Figure 1). Analogously as in the previous case we can assume that the strips $S_{1}$ and $S_{2}$ have width $\frac{(p_{1}+1)\sqrt{1+d}}{2}$ and $S_{3}$ has width $p_{1}+1$.

Now we will compute the area of the intersection of the strips $S_{1}$ and $S_{2}$. Denote $F=S_{1}\cap S_{2}$. The polygon $F$ is a rhombus with height (i.e., the distance between the opposite sides) $h=(p_{1}+1)\frac{\sqrt{1+d}}{2}$. Denote by $a=|A_{1}A_{2}|$ the length of the side of the rhombus $F$ and by $\alpha $ the argument of the number $\frac{1+\sqrt{-d}}{2}$. Then the area of the rhombus $F$ is given by the following formula: $$P_{F}=2\cdot\frac{1}{2}a^{2}\sin (\pi -2\alpha )=\frac{h^{2}}{2\sin \alpha \cos \alpha }=\frac{(p_{1}+1)^{2}\frac{(1+d)}{4}}{2\cdot\frac{\sqrt{d}}{1+d}}=\frac{(p_{1}+1)^{2}(1+d)^{2}}{8\sqrt{d}}.$$ Denote by $E$ the intersection of $F$ and the strip $S_{3}$. Now we will compute the area of $E$. Denote by $b_{1}$ the length of the horizontal diagonal $A_{1}A_{3}$ of $F$ and by $b_{2}$ the length of the vertical diagonal $A_{2}A_{4}$ of $F$. Note that $d_2=d_1/\sqrt{d}$ and $P_F=\frac{1}{2} d_1 d_2=\frac{1}{2} \frac{d_1^2}{\sqrt{d}}.$ Since the length $b_{1}$ is equal to $\frac{(p_{1}+1)(d+1)}{2}$, $C< \frac{p_{1}+1}{2}$ and $C< \frac{b_{1}}{2}-\frac{p_{1}+1}{2}$, the domain $E$ is a hexagon. Denote by $E_{1}$ and $E_{2}$ the connected components of the set $F\setminus E$. The sets $E_{1}$ and $E_{2}$ are triangles. Denote by $P_{E_{1}}$ (resp. $P_{E_{2}}$) the area of the triangle $E_{1}$ (resp. $E_{2}$). Assume that $P_{E_{1}}\leqslant P_{E_{2}}$. Since the distance of $A_1$ to $B_1 B_6$ is $(p_1+1)(d+1)/4-(p_1+1)/2-x$ and the distance of $A_3$ to $B_3 B_4$ is $(p_1+1)(d+1)/4-(p_1+1)/2+x$, we have the following formulae: \begin{align*}P_{E_{1}}=\frac{\Big(\frac{(p_1+1)(d-1)}{4}-x\Big)^{2}}{\sqrt{d}}, &  & P_{E_{2}}=\frac{\Big(\frac{(p_1+1)(d-1)}{4}+x\Big)^{2}}{\sqrt{d}}.\end{align*}The area of the hexagon $E$ is given by the  formula \begin{align*}P_{E}=P_{F}-P_{E_{1}}-P_{E_{2}}=\frac{(p_1+1)^{2}((1+d)^{2}-(1-d)^2)-16x^{2}}{8\sqrt{d}}=\frac{(p_1+1)^{2}d-4x^{2}}{2\sqrt{d}}.\end{align*} It follows that the hexagon $E$ contains $P_{E}\cdot\frac{1}{\frac{\sqrt{d}}{2}}+o(P_{E}\cdot\frac{1}{\frac{\sqrt{d}}{2}})$ points from $\mathcal{O}_{K}$. We compute $$P_{E}\cdot\frac{1}{\frac{\sqrt{d}}{2}}=\frac{(p_1+1)^{2}d-4x^{2}}{d}=(p_1+1)^{2}-\frac{4x^{2}}{d}.$$ This means that the hexagon $E$ contains at most $p_1^{2}+o(p_1^{2})$ points from $\mathcal{O}_{K}$. 

\begin{figure}[hb]
  \centering

\includegraphics[width=140mm]{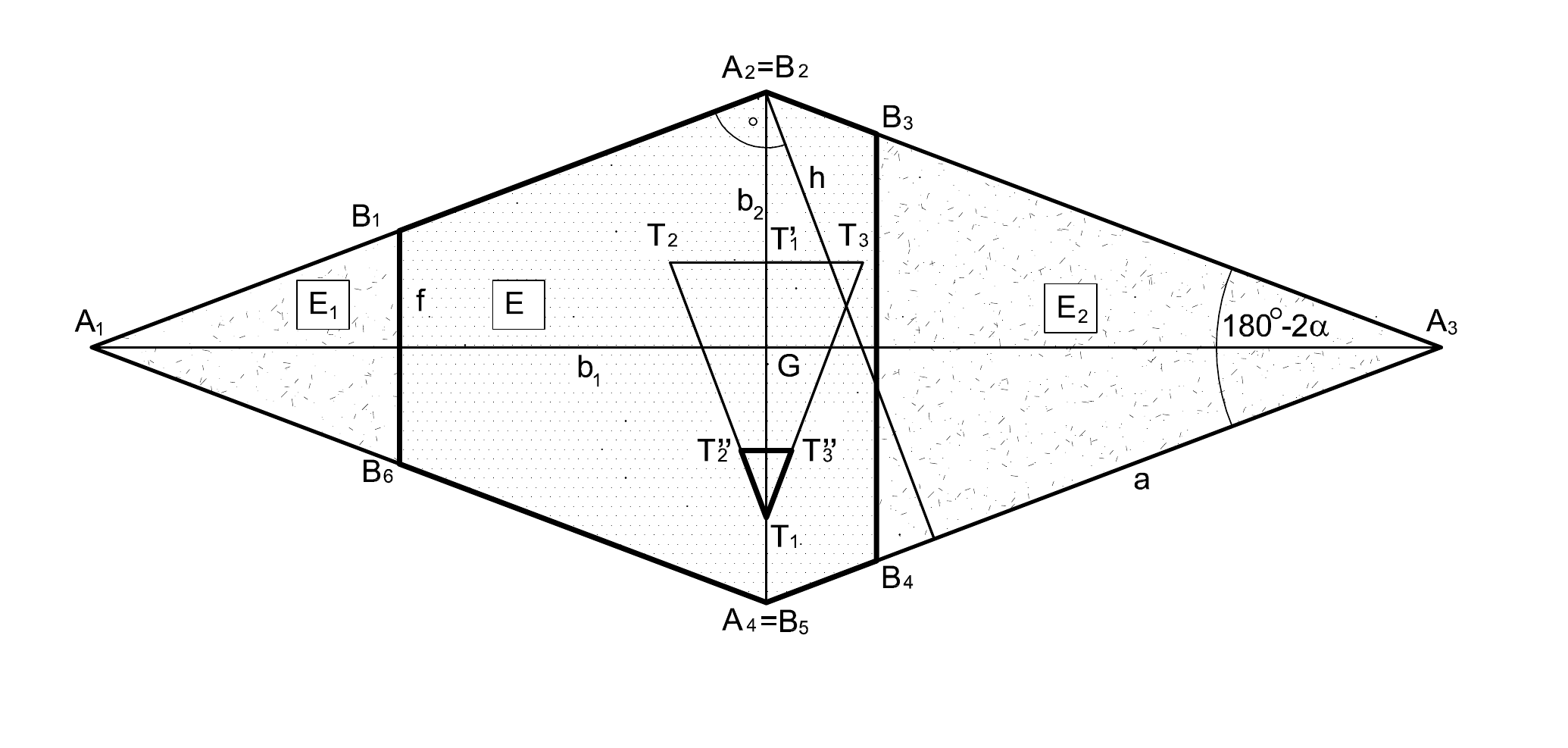}
\caption{The triangles $T$ and $U_1$. Considering these triangles leads to an improved bound on the cardinality of $S$.}
\end{figure}

We will now show that some of these points cannot belong to the set $S$. Denote by $G$ the intersection of the diagonals of the rhombus $F$. Consider a triangle $T$ with vertices $T_{1}, T_{2}, T_{3}$ and denote by $T'_{1}$ the orthogonal projection of the point $T_{1}$ onto the side $T_{2}T_{3}$. We pick the vertices of $T$ so that  $T_{2}T_{3}$ is horizontal, $|T_{2}T_{3}|=C_{1}p_{2}$, $|T_{1}T_{3}|=|T_{1}T_{2}|=\frac{C_{1}p_{2}\sqrt{1+d}}{2}$, $|GT_{1}|=2|GT'_{1}|$, the imaginary part of the point $T_{1}$ is smaller than the imaginary part of the point $T_{2}$ and the real part of the point $T_{2}$ is smaller than the real part of the point $T_{3}$ (see Figure 2). Denote by $f$ the length of the smallest vertical side of the hexagon $E$. We compute: $b_{2}=\frac{(p_{1}+1)(d+1)}{2\sqrt{d}}$, $|T_{1}T'_{1}|=\frac{C_{1}p_{2}\sqrt{d}}{2}$ and $f=\frac{(p_{1}+1)(d-1)-4x}{2\sqrt{d}}$. As $|T_{1}G|< \frac{b_{2}}{2}$, $|T'_{1}G|< \frac{f}{2}$ and $|T_{2}T_{3}|< p_{1}+1-2x$, the triangle $T$ is contained in the hexagon $E$ for $\epsilon $ small enough and $n$ large enough (to see that, recall that $\frac{3\sqrt{2}}{4(1+2\epsilon )}<C_{1}< \frac{3\sqrt{2}}{4(1+\epsilon )}$ and $0<C<\frac{p_{1}+1-C_{1}p_{2}}{2}$). Consider a triangle $U_{1}$ with vertices $T_{1}, T''_{2}, T''_{3}$ which is the image of a triangle $T$ under the homothety transformation with center in $T_{1}$ and ratio $\frac{C_{1}-1}{C_{1}}$. The area $P_{U_1}$ of the triangle $U_{1}$ is equal to $\frac{(C_{1}-1)^{2}p_{2}^{2}\sqrt{d}}{4}.$ This implies that the triangle $U_{1}$ contains $P_{U_{1}}\cdot\frac{1}{\frac{\sqrt{d}}{2}}+o(P_{U_{1}}\cdot\frac{1}{\frac{\sqrt{d}}{2}})=\frac{(C_{1}-1)^{2}p_{2}^{2}}{2}+o(\frac{(C_{1}-1)^{2}p_{2}^{2}}{2})$ points from $\mathcal{O}_{K}$. Let $y\in U_{1}\cap \mathcal{O}_{K}$. Consider a set $T_{y}=\lbrace y, y+p_{2}\frac{1+\sqrt{-d}}{2}, y+p_{2}\frac{-1+\sqrt{-d}}{2} \rbrace$. The set $T_{y}$ is contained in the hexagon $E$ (for large enough $n$ and small enough  $\epsilon $). In fact, $U_1$, $U_1+p_{2}\frac{1+\sqrt{-d}}{2}$ and $U_1+p_{2}\frac{-1+\sqrt{-d}}{2}$ are three small triangles that are lying inside $T$ and share one vertex with it.  The elements of $T_{y}$ give the same remainder modulo $p_{2}$. As there are $p_{2}^{2}> \frac{n}{2}$ different remainders modulo $p_{2}$, by Corollary \ref{p.n-univCondition} the set $S$ cannot contain three different elements which give the same remainder modulo $p_{2}$. Therefore, at least one element of the set $T_{y}$ does not belong to the set $S$. Considering $T_{y}$ for every $y\in U_{1}\cap \mathcal{O}_{K}$, we get that at least $\frac{(C_{1}-1)^{2}p_{2}^{2}}{2}+o(\frac{(C_{1}-1)^{2}p_{2}^{2}}{2})$ points from $E\cap \mathcal{O}_{K}$ do not belong to the set $S$ (for $\epsilon $ small enough and large enough $n$). This implies that $|S|\leq p_{1}^{2}-\frac{(C_{1}-1)^{2}p_{2}^{2}}{2}+o(p_{1}^{2})-o(p_{2}^{2})\leq ((1+\epsilon)^2-\frac{(C_1-1)^2}{4})n+o(n)<n$ (for $\epsilon $ small enough and large enough $n$). By Lemma \ref{dolne}, this gives a contradiction and ends the proof of this subcase.

\textbf{Subcase 2.} Assume $x> C$. Using the calculations from the previous subcase we see that if the intersection of the strips $S_{1}, S_{2}$, and $S_{3}$ is a hexagon or a triangle, then it contains less than $p_{1}^{2}+ o(p_{1}^{2})$ points from $\mathcal{O}_{K}$ (for $\epsilon $ small enough and large enough $n$). Again, by Lemma \ref{dolne}, this gives a contradiction. It remains to consider the case when the intersection of the strips $S_{1}$, $S_2$, and $S_3$ is empty, but this obviously gives a contradiction. This ends the proof of Theorem \ref{3}.

\end{proof}

\section{Construction of $n$-universal sets with $n+2$ elements}
\begin{theorem}
Let $A$ be a Dedekind domain. Then for any non-negative $n$ there exists an $n$-universal set $E_n \subset A$ with $n+2$ elements. In fact, one can construct an increasing family $E_0\subset E_1\subset E_2 \subset \ldots$ of $n$-universal sets $E_n$ in $A$ with $n+2$ elements.
\end{theorem}
\begin{proof}
We will construct sets $E_n$ inductively. Set $E_0=\{0,1\}$. Suppose we have constructed an $n$-universal set $E_n$ with $n+2$ elements. We will show that by adding an appropriate element we can extend $E_n$ to an $(n+1)$-universal set $E_{n+1}$. Note that for almost every prime ideal $\mathfrak{p}$ of $\mathcal{O}_{K}$ the set $E_{n}$ is already almost uniformly distributed modulo powers of $\mathfrak{p}$. This is the case for all prime ideals $\mathfrak{p}$ which do not contain any difference of two elements from $E_{n}$. Therefore, in order to construct the set $E_{n+1}$, we only need to look at a finite set of prime ideals $S_n=\{\mathfrak{p}\mid \frak p\supset Vol(E_n)\}$. For any such prime ideal $\mathfrak{p}\in S_{n}$, there exists a subset of $E_{n}$ with $n+1$ elements which is almost uniformly distributed modulo powers of $\mathfrak{p}$. We can extend such a subset by one element $x_{\mathfrak{p}}$ so that the extended set remains almost uniformly distributed. Denote by $\nu_{\mathfrak{p}}$ the highest power of a prime ideal $\mathfrak{p}\in S_n$ dividing a difference of a pair of elements in $E_n\cup \{x_\frak p\}$. Now, by Chinese Remainder Theorem we can find an element $x_n$ in $\mathcal{O}_{K}$ such that $x_n \equiv x_{\mathfrak{p}}\pmod{\mathfrak{p}^{\nu_{\mathfrak{p}}+1}}$ for each $\mathfrak{p}\in S_{n}$. Put $E_{n+1}=E_{n} \cup \{x_n\}$. Then the set $E_{n+1}$ contains for any prime $\mathfrak{p}$ a subset with $n+2$ elements which is almost uniformly distributed modulo powers of $\mathfrak{p}$. Hence $E_{n+1}$ is $(n+1)$-universal.
\end{proof}
Consider this construction in the case when $A=\mathcal{O}_K$ is the ring of integers in a number field $K$. In the $n$-th step of this construction, we need to add an element $x_n$ satisfying a simultaneous congruence modulo all primes (and powers of primes) modulo which $E_n$ fails to be almost uniformly distributed. These are exactly the prime powers dividing $Vol(E_n)/\prod_{m=1}^{n+1} m!_K$. For future reference, let us denote the set of these prime powers by $S_n'$. In the proof, we have used the Chinese Remainder Theorem to justify that such $x_n$ always exists. In practice, in order to find $x_n$, we need to solve a system of congruences $x_n\equiv x_{\frak{p}} \pmod{\frak{p}^k}$ for $p^k \in S_n'$. The time needed to solve such a system of congruences is polynomial in the logarithm of the product of norms of ideals in $S_n'$. Consequently, the time needed to find $x_n$ is polynomial in $\log N(Vol(E_n)/\prod_{m=1}^{n+1} m!_K)$. It follows that in order to have a good control on the running time of our construction, we have to control the growth of $U_n=N(Vol(E_n)/\prod_{m=1}^{n+1} m!_K)$. Even for $A=\Z$, this task proves to be difficult, because the solution of a system of congruences modulo elements of $S_n'$ cannot be, in general, bounded by anything less than the product of all numbers in $S_n'$. In the case of $\mathcal{O}_K$, we cannot bound the norm of the solution by anything significantly smaller than the product of norms of elements in $S_n'$. An easy computation shows that if in each step the norm of $x_n$ is of order $U_n$ and for at least one $n_0$ the number $U_{n_0}$ is big enough, then the bound that we will get on the consecutive $U_n$'s is an exponential tower of height $2$, i.e., $U_n\ll \exp(a\exp(bn))$ for some $a,b>0$. The expected time needed to find $x_n$ is therefore polynomial in $\log U_n\ll \log(\exp(a\exp(bn))=a\exp(bn)$. The latter bound is exponential in $n$ which means that this construction is likely to be impractical for large $n$. In Section \ref{s.Probabil}, we present an alternative probabilistic method that may be used to construct $(n+d)$-universal sets for $d=[K:\Q]$.

\section{Euler-Kronecker constants}\label{s.EKconstant}
The problem of existence of $n$-optimal sets  in general number fields seems  much harder than in the imaginary quadratic case. In the proof of Theorem \ref{t.ImQuad}, we relied on the fact that collapsing the set reduces its volume, hence an $n$-optimal set is necessarily collapsed with respect to every direction. This technique requires the norm $N_{K/\Q}$ to be convex, which (essentially by Dirichlet's unit theorem) holds only when $K=\Q$ or $K$ is an imaginary quadratic number field. In this section, we compute the asymptotic growth of the norm of the volume of $n$-universal sets. As a by-product of our attempts to prove non-existence of $n$-optimal sets in arbitrary number fields, we obtain a lower bound on the values of Euler-Kronecker constants.

Let us recall the definition of Euler-Kronecker constants.
\begin{definition}\cite{Ihara1}
Let $\zeta_K(s)$ be the Dedekind zeta function of a number field $K$. Let 
$$\zeta_K(s)=\frac{c_{-1}}{(s-1)}+c_0+c_1(s-1)+\ldots$$ be the Laurent expansion of $\zeta_K$ at $s=1$. We define the \textbf{Euler-Kronecker constant} $\gamma_K$ as the quotient $c_0/c_{-1}$, or equivalently as the constant term of the Laurent expansion of $\zeta_K'/\zeta_K$ at $s=1$.
\end{definition}
For $K=\Q$, the constant $\gamma_\Q$ is the Euler-Mascheroni constant $\gamma$ given by the formula
$$\gamma_{\Q}=\lim_{n\to\infty}\left(\sum_{i=1}^n\frac{1}{i}-\log n\right).$$ For more  information on the Euler-Kronecker constants, we refer the reader to the article of Ihara \cite{Ihara1}. Euler-Kronecker constants arise in our considerations via the following theorem due to M. Lamoureux \cite{MLam14}.
\begin{theorem}\textup{(}\cite[Theorem 1.2.4]{MLam14}\textup{)}
$$\log n!_K=n\log n-n(1+\gamma_K-\gamma_\Q)+o(n).$$
\end{theorem}
Using Proposition \ref{p.NunivVol}, we immediately obtain the following corollary.
\begin{corollary}\label{c.FactorialEst}
Let $S$ be an $n$-optimal subset of $\mathcal{O}_K$. Then 
$$\log N(Vol(S))={n^2}\log n-\frac{n^2}{2}-{n^2}(1+\gamma_K-\gamma_\Q)+o(n^2).$$
Moreover, for every subset $S'\subset \mathcal{O}_K$ with $n+1$ elements we have %popr
$$\log N(Vol(S'))\geq {n^2}\log n-\frac{n^2}{2}-{n^2}(1+\gamma_K-\gamma_\Q)+o(n^2).$$
\end{corollary}
One can try to prove  non-existence of $n$-optimal subsets in number fields by combining the above estimate with a lower bound on $Vol(S)$ obtained by some geometric arguments. This proves to be problematic in fields which have an infinite group of integral units due to  non-convexity of the norm.
However, one can use the estimates from Corollary \ref{c.FactorialEst} to obtain the following analytic theorem.
\begin{theorem}\label{t.LogIneq}
Let $U$ be an open bounded subset of $\R^d$. For any $x=(x_1,\ldots, x_d)\in \R^d$, write $\|x\|=\prod_{i=1}^d|x_i|$. (Note that this notation is nonstandard as $\|\cdot\|$ is not a norm on $\R^d$.)
 Then 
$$\int_U\int_U\log \|x-y\|dxdy\geq m(U)^2(c_d+\log m(U)),$$
where $c_d>0$ is a constant depending only on $d$ and $m(U)$ is the Lebesgue measure of $U$. 
\end{theorem}
\begin{proof}
Let us first treat the case $m(U)=1$. We will reduce the proof to this case by scaling. Choose some totally real number field $K$ of degree $d$ and let $\sigma_1,\ldots, \sigma_d$ be all the embeddings $K\to \R$. Identify the ring of integers $\mathcal{O}_K$ of $K$ with the lattice $L$ in $\R^d$ by means of the Dirichlet embedding $K \to \R^d$, $x\mapsto (\sigma_i(x))_{i=1,\ldots,d}$. For a natural number $n$, let $A_n= |\frac{1}{n}L\cap U|$. Consider the normalized counting measures $\mu_n$ on the sets $(\frac{1}{n}L\cap U)^2$. The sequence $\mu_n$ is a sequence of probability measures which converges in the weak-* topology to the Lebesgue measure on $U\times U$ as $n\to\infty$. Choose $M<0$. Then the function $U\times U \to \R$, $(x,y)\mapsto \max\{M,\log\|x-y\|\}$ is continuous and bounded, so by the weak-* convergence of measures we get 
\begin{align*}\int_U\int_U \max\{M,\log\|x-y\|\}dxdy =&
\lim_{n\to\infty}\frac{1}{A_n^2}\sum_{x,y\in \frac{1}{n}L\cap U}\max\{M,\log\|x-y\|\} \geq\\ & \lim_{n\to\infty}\frac{1}{A_n^2}\sum_{\substack{x,y\in \frac{1}{n}L\cap U \\ x\neq y}} \log\|x-y\|.
\end{align*}
The rightmost term can be rewritten as 
\begin{align*}
&\frac{1}{A_n^2}\sum_{\substack{x,y\in \frac{1}{n}L\cap U \\ x\neq y}} \log\|x-y\|=\\
&\frac{1}{A_n^2}\sum_{\substack{x,y\in L\cap nU\\x\neq y} }(\log\|x-y\|- d\log n)=\\
&\frac{1}{A_n^2}\sum_{\substack{x,y\in L\cap nU\\x\neq y}}\log | N(x-y)|- \frac{d(A_n-1)}{A_n} \log n =\\
&\frac{1}{A_n^2}\log N(Vol(L\cap nU))-d\log n+\frac{d}{A_n}\log n.
\end{align*}
By Corollary \ref{c.FactorialEst}, we bound the last expression from below by \begin{align} &\geq\frac{1}{A_n^2}(A_n^2\log A_n-\frac{A_n^2}{2}-A_n^2(1+\gamma_K-\gamma_\Q)) -d\log n+\frac{d}{A_n}\log n+o(1)\\ &= \log A_n-\frac{3}{2}-\gamma_K+\gamma_\Q-d\log n+\frac{d}{A_n}\log n+o(1).\label{dpa123}\end{align} Since $U$ is an open set of measure $1$ and $L$ is a lattice in $\R^d$ of covolume $\sqrt{|\Delta_K|}$, $A_n=\frac{n^d}{\sqrt{|\Delta_K|}}+o(n^d)$. Hence, $$\log A_n = d\log n-\frac 1 2 \log |\Delta_K| + o(1)$$ and $\frac{d}{A_n}\log n=o(1)$. Therefore, we can bound the expression (\ref{dpa123}) from below by
\begin{align*}
&-\frac{1}{2}\log |\Delta_K|-\frac 3 2 -\gamma_K+\gamma_{\Q}+o(1).
\end{align*}
Let us define the constant $c_{d,K}=-\frac{3}{2}-\gamma_K+\gamma_Q-\frac 1 2 \log |\Delta_K|$ and set $c_d=\sup c_{d,K}$, where supremum is taken over all totally real fields $K$ of degree $d$. Then the inequality above shows that 
$$\int_U\int_U \max\{M,\log\|x-y\|\}dxdy\geq c_d.$$ But 
$$\lim_{M\to -\infty}\int_U\int_U \max\{M,\log\|x-y\|\}dxdy=\int_U\int_U \log\|x-y\|dxdy,$$ which finishes the proof in the case $m(U)=1$.
It remains to reduce the proof to this case. Let $U\subset \R^d$ be any open bounded subset. For any $\lambda>0$, let $U_{\lambda}$ denote the set $U$ scaled by a factor of $\lambda$. Integrating by substitution, we get
\begin{align*}\int_{U_\lambda}\int_{U_\lambda}\log \|x-y\|dxdy= &\lambda^{2d}\int_U\int_U \log\|\lambda(x-y)\|dxdy\\ =&\lambda^{2d}\int_U\int_U(\log\|x-y\| + d\log{\lambda}) dxdy.\end{align*}
For an arbitrary set $U$ of finite measure $m(U)$, we can scale it by a factor of $\lambda=m(U)^{-1/d}$ so that it becomes of measure one and apply the known result to this measure one set. This gives
$$\int_{U}\int_{U}\log \|x-y\|dxdy\geq m(U)^2(c_d+\log m(U)).\qedhere$$
\end{proof}
As a corollary we obtain a lower bound on the value of $\gamma_K$. 
\begin{corollary}Let $K$ be a totally real number field. Then
$$\gamma_K\geq -\frac 1 2 \log |\Delta_K| +\frac 3 2 d - \frac 3 2 + \gamma_\Q.$$
\end{corollary}
\begin{proof}
Let $U=(0,1)^d\in\R^d$. Then $m(U)=1$ and by Theorem \ref{t.LogIneq}, we have 
$$\int_U\int_U\log \|x-y\|dxdy=d\int_0^1\int_0^1\log |x-y|dxdy=-\frac{3}{2}d\geq c_d\geq c_{d,K}.$$
Hence, by the definition of $c_{d,K}$ we get
$$-\frac{3}{2} d\geq -\frac 3 2 -\gamma_K+\gamma_Q-\frac 1 2 \log |\Delta_K|.$$
The desired inequality easily follows.
\end{proof}
The main term of the last inequality is $\log\sqrt{|\Delta_K|}$. A very similar, but slightly stronger inequality has been proved by Ihara in \cite{Ihara1}. More precisely, he obtained the following bound (\cite[Proposition 3]{Ihara1})  $\gamma_K > -\frac 1 2 \log |\Delta_K | + \frac{\gamma_{\Q}+\log(4\pi)}{2}  d - 1$. The coefficient at $d$ is $(\gamma_{\Q}+\log(4\pi))/2 = 1.554\ldots$, so our bound is slightly weaker.

\section{A probabilistic construction}\label{s.Probabil}
In this section, we present a probabilistic method to construct an $n$-universal subset of $\mathcal{O}_K$ with $n+d$ elements where $d=\dim_\Q K$. 
\subsection{Outline of the construction.}
Fix a natural number $m\geq 1$. In order to choose a random subset $S$ of the ring $\mathcal{O}_K$ with $m+n=k$ elements, we fix $k$ independent random variables $\xi_1,\ldots,\xi_k$ with values in $\mathcal{O}_K$ and set $S=\{\xi_1,\ldots,\xi_k\}$. By Corollary \ref{p.n-univCondition}, in order to check if $S$ is $n$-universal, all we need to do is to verify that for every prime ideal $\frak p$ of $\mathcal{O}_K$, $S$ contains $n+1$ elements which are almost uniformly distributed modulo powers of $\frak p$. As we have pointed out in Definition \ref{d.AUD}, the condition simplifies when $|\mathcal{O}_K/\frak p|\geq n+1$. The set $S$ contains a subset of size $n+1$ which is almost uniformly distributed modulo powers of $\frak p$ if and only if $|\pi_{\frak p}(S)|\geq n+1$, where $\pi_{\frak p}:\mathcal{O}_K\to \mathcal{O}_K/\frak p$ is the canonical projection. The probability of this event depends only on the variables  $\pi_{\frak p}(\xi_1),\ldots,\pi_{\frak p}(\xi_k)$. Thus, it can be estimated relatively easy. This suggests the following strategy to choose a random subset of $\mathcal{O}_K$:
\begin{itemize}
\item Fix a probability measure $\mu$ on $\mathcal{O}_K$.
\item Fix $L>2(n+1)$ (for technical reasons it is more convenient to assume $L>2n+2$ rather than just $L>n+1$), and choose elements $a_1,\ldots,a_k$ such that $\{a_1,\ldots,a_{n+1}\}$ is almost uniformly distributed modulo powers of $\frak p$ for every prime ideal $\frak p$ with $|\mathcal{O}_K/\frak p|\leq L$. Again for technical reasons we choose $a_1,\ldots,a_k$ such that $\{a_1,\ldots,a_k\}$ has at least $n+1$ distinct elements modulo $L!$.
\item Let $\psi_1,\ldots,\psi_k$ be independent random $\mathcal{O}_K$-valued variables with distribution $\mu$.
Set $\xi_i=a_i+L!\psi_i$. Then the random set $S=\{\xi_1,\ldots,\xi_k\}$ contains a subset of size $n+1$ which is almost uniformly distributed modulo every power of a prime ideal $\frak p$ satisfying $N(\frak p)\leq L$ \footnote{This is the contents of Lemma \ref{l.Expliq}.}. Thus, in order to compute the probability that $S$ is $n$-universal we only need to estimate the probabilities that $|\pi_{\frak p}(S)|\geq n+1$ for every prime ideal with norm bigger than $L$.
\end{itemize}
The reasons why we have imposed these technical conditions will become clear during the proof.
We have outlined the general way in which we can construct our set. %popr 
Until now we have made no reference to a random walk on $\mathcal{O}_K$, but we shall use them to construct the measure $\mu$. This will allow us to interpret the elements of $S$ as steps of $n+d$ independent random walks on $\mathcal{O}_K$. In the following section, we prove some estimates on the probability that the set $S$ constructed in the way described above is $n$-universal. 
\begin{lemma}\label{l.Expliq}
The set $S=\{\xi_1,\ldots,\xi_k\}$, constructed as in the outline, always contains a subset of size $n+1$ which is almost uniformly distributed modulo powers of prime ideals of norm not exceeding $L$.
\end{lemma}
\begin{proof}
We have to show that $S=\{\xi_1,\ldots,\xi_k\}$ contains an $(n+1)$-element subset that is  almost uniformly distributed modulo every power of a prime ideal $\frak p$ such that  $|\mathcal{O}_K/\frak p|\leq L$. For a fixed ideal $\frak p$, the existence of an almost uniformly distributed subset of size $n+1$ depends only on the residues of $S$ modulo ${\frak p}^m$ where $m$ is the smallest integer for which $N({\frak p}^m)=N(\frak p)^m\geq n+1$. Indeed, a set with $n+1$ elements is almost uniformly distributed modulo ${\frak p}^m$ if and only if it has distinct elements modulo ${\frak p}^m$, but then it is also uniformly distributed modulo all higher powers of $\frak p$. By construction, we know that $\{a_1,\ldots,a_{n+1}\}$ is almost uniformly distributed modulo powers of $\frak p$ if $N(\frak p)\leq L$. 
Adding multiples of $L!$ to the elements of this set cannot change that because $L!$ is divisible by ${\frak p}^m$. Indeed, $(N \frak{p})^{m-1}<n+1$,
$N(\frak p)\leq L$ and $L>2(n+1)$, so $L!$ is divisible by $(N \frak{p})^m$ and hence also by $\frak{p}^m$.
\end{proof}
\subsection{Preliminary results}
Throughout this and the following sections, $\mu^{*n}$ will denote the $n$-th convolution of a measure $\mu$. We will write $\widehat G$ for the group of characters of $G$. We shall often treat  probability measures on discrete sets as functions and whenever we write $\mu(x)$ we mean $\mu(\{x\})$. Let us recall the definition of the Fourier transform on finite abelian groups.
\begin{definition} Let $\mu$ be a probability measure on a finite abelian group $G$. The \textbf{Fourier transform} of $\mu$ is defined as the measure $\hat{\mu}$ on $\hat{G}$ given by
$$\hat{\mu}(\chi)=\sum_{g\in G}\overline{\chi(g)}\mu(g)$$ for  $\chi\in\widehat{G}$.
\end{definition}
A classical property of the Fourier transform is that the transform of a convolution is the product of transforms, i.e., $\widehat{\mu*\nu}=\widehat{\mu}\widehat{\nu}$. The inverse Fourier transform is given by the formula
$$\mu(g)=\frac{1}{|G|}\sum_{\chi\in\widehat{G}}\chi(g)\widehat{\mu}(\chi).$$
The following series of lemmas will be used to estimate the probability that a randomly chosen subset of $\mathcal{O}_K$ contains $n+1$ elements which are almost uniformly distributed modulo a certain ideal $I$ with $|\mathcal{O}_K/I|\geq n+1$. We call a family $A=\{A_1,A_2,\ldots,A_k\}$ of subsets of $\{1,2,\ldots, n+m\}$ an $m$-partition if it satisfies the following conditions:
\begin{itemize}
\item  $A_i\neq \emptyset$, $A_i\cap A_j=\emptyset$ if $j\neq i$, 
\item $\bigcup A_i=\{1,\ldots,m+n\}$,
\item $\sum_{i=1}^k(|A_i|-1)=m$ or, equivalently, $k=n$.
\end{itemize}
The following lemma allows us to estimate the probability that a collection of $n+m$ independent random variables forms a set of cardinality at most $n$.
\begin{lemma}\label{l.ProbEst1}
Let $X$ be a finite set, $n,m$ natural numbers and $\xi=(\xi_1,\xi_2,\ldots,\xi_{n+m})$ independent random variables with values in $X$ with  probability distributions respectively $\mu_1,\mu_2,\ldots,\mu_{n+m}$. Let $P=P_{\xi,m,n}$ denote the probability that the set $\{\xi_1,\xi_2,\ldots,\xi_{n+m}\}$ contains at most $n$ distinct elements. Then $$P\leq \sum_{A {\rm\ \ m-partition}}\prod_{i=1}^k\left(\sum_{x\in X}\frac{1}{|A_i|}\sum_{j\in A_i}\mu_i(x)^{|A_i|}\right).$$
\end{lemma}
\begin{proof}
For any $m$-partition $A=\{A_1,A_2,\ldots,A_k\}$ of $\{1,2,\ldots, n+m\}$, let $R_A$ denote the event that the random function $i\mapsto\xi_i$ is constant on each $A_i$. The variables $\xi_i$ are independent, so we have 
$$\mathbb{P}[R_A]=\prod_{i=1}^k\left(\sum_{x\in X}\prod_{j\in A_i}\mu_j(x)\right).$$
The random set $\{\xi_1,\xi_2,\ldots,\xi_{n+m}\}$ has no more than $n$ distinct elements if and only if there exists an $m$-partition $A$ such that the event $R_A$ happened. That gives us the following inequality:
$$P\leq\sum_{A {\rm\ \ m-partition}}\mathbb{P}[R_A]=\sum_{A {\rm\ \ m-partition}}\prod_{i=1}^k\left(\sum_{x\in X}\prod_{j\in A_i}\mu_j(x)\right).$$
Using the inequality between arithmetic and geometric means, we get $$\prod_{j\in A_i}\mu_j(x)\leq \frac{1}{|A_i|}\sum_{j\in A_i}\mu_j(x)^{|A_i|}.$$ Hence
$$P \leq \sum_{A {\rm\ \ m-partition}}\prod_{i=1}^k\left(\sum_{x\in X}\frac{1}{|A_i|}\sum_{j\in A_i}\mu_j(x)^{|A_i|}\right).$$\qedhere
\end{proof}
\begin{proposition}\label{p.ProbEst2}
Let $G$ be a finite group, $n,m$ positive integers and  let $\xi=(\xi_1,\xi_2, \ldots,\xi_{n+m})$ be independent random elements with distributions respectively $\mu_1,\mu_2,\ldots,\mu_{n+m}$. Let $\mu$ be a probability measure on $G$. We assume that all $\mu_i$ are some translates of $\mu$, i.e., we have $\mu_i(x)=\mu(s_ix)$ for $i=1,\ldots, n+m$ and some fixed elements $s_i\in G$. Let $P=P_{\xi,m,n}$ denote the probability that the set $\{\xi_1,\xi_2,\ldots,\xi_{n+m}\}$ contains at most $n$ distinct elements.%popr
 Then 
$$P\ll\left(\frac{\sum_{\chi\in\widehat{G}}|\widehat{\mu}(\chi)|}{|G|}\right)^m$$
with the implicit constant depending only on $n$ and $m$, but not %popr
 on $G$ or $\mu_i$.
\end{proposition}
For the proof we shall need the following variation of the Haussdorf-Young inequality. 
\begin{lemma}\label{l.HYineq}Let $G$ be a finite abelian group, let $p,q$ be real numbers such that $\frac{1}{p}+\frac{1}{q}=1$ and $1< p\leq 2$. %popr
 Then for any $f:G\to\C$ we have
$$\sum_{g\in G}|f(g)|^q\leq\left(\frac{1}{|G|}\sum_{\chi\in\widehat{G}}|\widehat{f}(\chi)|^p\right)^{q-1}.$$
\end{lemma}
\begin{proof}
By taking the $q$-th root of both sides, we see that the inequality is equivalent to  the inequality
$$\|f\|_{l^q(G)}\leq\|\widehat{f}\|_{L^p(\widehat{G})}$$
where $\widehat{G}$ is endowed with a normalized counting measure. The group $G$ is canonically isomorphic to the dual group of $\widehat{G}$ and the function $\overline{f}$ is equal to the Fourier transform of $\overline{\widehat{f}}$. Thus, the equality may be rewritten as 
$$\|\widehat{\overline{\widehat{f}}}\|_{l^q(\widehat{\widehat{G}})}\leq\|\overline{\widehat{f}}\|_{L^p(\widehat{G})}$$ which is the Hausdorff-Young inequality (\cite[ Chapter 9.5]{HLP52}) applied to the group $\widehat{G}$ and the function $\overline{\widehat{f}}$.
\end{proof}

\begin{proof}[Proof of Proposition \ref{p.ProbEst2}]
By Lemma \ref{l.ProbEst1} we have 
$$P\leq  \sum_{A {\rm\ \ m-partition}}\prod_{i=1}^k\left(\sum_{g\in G}\frac{1}{|A_i|}\sum_{j\in A_i}\mu_j(g)^{|A_i|}\right).$$
Using Lemma \ref{l.HYineq} with $q=|A_i|,p=|A_i|/(|A_i|-1)$ for any $A_i$ with $|A_i|\geq 2$, we get
\begin{equation}\label{e.Eq1}
\sum_{g\in G}\frac{1}{|A_i|}\sum_{j\in A_i}\mu_j(g)^{|A_i|}\leq\frac{1}{|A_i|}\sum_{j\in A_i}\left(\frac{1}{|G|}\sum_{\chi\in\widehat{G}}|\widehat{\mu_j}(\chi)|^{|A_i|/(|A_i|-1)}\right)^{|A_i|-1}.
\end{equation}
Since $\mu_j$ is a probability measure, we always have $|\widehat{\mu_j}(\chi)|\leq 1$, so 
$$|\widehat{\mu_j}(\chi)|^{|A_i|/(|A_i|-1)}\leq |\widehat{\mu_j}(\chi)|.$$ Moreover, since $\mu_j(x)=\mu(s_jx)$, we have $\widehat{\mu_j}(\chi)=\chi(s_j)\widehat{\mu}(\chi)$, so $|\widehat{\mu_j}(\chi)|=|\widehat{\mu}(\chi)|$ for all $j$. We conclude from these remarks and inequality (\ref{e.Eq1}) that
$$\sum_{g\in G}\frac{1}{|A_i|}\sum_{j\in A_i}\mu_j(g)^{|A_i|}\leq \left(\frac{1}{|G|}\sum_{\chi\in\widehat{G}}|\widehat{\mu}(\chi)|\right)^{|A_i|-1}.$$ Note that this holds even if $|A_i|=1$ because then both sides are equal to $1$. Thus, 
$$P\leq  \sum_{A {\rm\ \ m-partition}}\prod_{i=1}^k\left(\frac{1}{|G|}\sum_{\chi\in\widehat{G}}|\widehat{\mu}(\chi)|\right)^{|A_i|-1}.$$ For any $m$-partition we have $\sum_{i=1}^k(|A_i|-1)=m$, so 
$$P\ll \left(\frac{\sum_{\chi\in\widehat{G}}|\widehat{\mu}(\chi)|}{|G|}\right)^m$$ where the implicit constant is the number of all $m$-partitions of $\{1,\ldots,n+m\}$.
\end{proof}
We conclude this subsection with a proposition that will be used to evaluate the probability that a randomly chosen set is not $n$-universal. Recall that $\mathcal{O}_K$ is the ring of integers in a finite extension field $K$ of $\Q$ of degree $d\geq 2$. Let us fix $d$ elements $\omega_1,\ldots,\omega_d \in \mathcal{O}_K$ such that $\mathcal{O}_K=\Z\omega_1\oplus\ldots\oplus\Z\omega_d\simeq \Z^d$. We define a norm $\|\cdot\|$  on $\mathcal{O}_K$ by letting $\|a\|=\max |a_i|$ for $a=a_1\omega_1+\ldots+a_d\omega_d$. Recall that $N(I)=|\mathcal{O}_K/I|$ is the norm of an ideal $I$, and that the function $\mathcal{O}_K \to \mathbf{N}$, $a\mapsto N(a\mathcal{O}_K)=|N_{K/\Q}(a)|$ grows as $O(\|a\|^d).$\footnote{Indeed, the norm $N_{K/\Q}(a)$ is the determinant of the map $K\to K, x\mapsto ax$. The entries of the matrix representing the multiplication by $a$ are linear in $a_1,\ldots,a_d$, thus $N_{K/\Q}(a)$ is a homogeneous polynomial of degree $d$ in variables $a_1,\ldots,a_d$.}
\begin{proposition}\label{p.ProbCrit}
Let $\mu$ be a probability measure on $\mathcal{O}_K$, $L$ an integer and let $S=\{\xi_1,\ldots,\xi_{n+m}\}$, where $\xi_i$ are independent random variables defined as in the outline. For any non zero prime ideal $\frak p\in\Spec \, \mathcal{O}_K$, let $\mu_{\frak p}$ denote the probability measure on $\mathcal{O}_K/\frak p$ obtained as the projection of $\mu$ and let $P=P_{S,n}$ be the probability that $S$ is not $n$-universal. Then there exist  constants $c>0$ and $\kappa_0>0$ (depending on the field $K$ and the numbers $L$, $n$, and $m$) and a constant $C_0$ depending only on $n$ and $m$ such that for any natural number $\kappa>\kappa_0$ we have
$$P\leq C_0\left(\sum_{L<N(\frak p)<c\kappa^d}\left(\frac{\sum_{\chi\in\widehat{\mathcal{O}_K/\frak p}}|\widehat{\mu_{\frak p}}(\chi)|}{N(\frak p)}\right)^m\right) +(n+m)\mu(\{a\in \mathcal{O}_K|\|a\|>\kappa\}).$$
\end{proposition}
\begin{proof}
By Proposition \ref{p.CritNormal}, the set $S$ is $n$-universal if and only if for every prime ideal $\frak p\in\Spec \, \mathcal{O}_K$ we can find a subset of $S$ with $n+1$ elements which is almost uniformly distributed modulo powers of $\frak p$. By construction, $S$ contains such a subset for every prime ideal $\frak p$ with $N(\frak p)\leq L$. Thus, we only need to verify that we can find such a subset for prime ideals $\frak p$ whose norm is bigger than $L$. For such an ideal $\mathfrak{p}$, the set $S$ contains an $(n+1)$-element subset that is  almost uniformly distributed modulo powers of $\frak p$ if and %popr
 only if $\pi_{\frak p}(S)$ has $n+1$ elements. Let $P_{\frak p}$ denote the probability that $\pi_{\frak p}(S)< n+1$. The probability of a union of events is smaller or equal to the sum of their probabilities, so we have the following inequality
$$P\leq \sum_{L<N(\frak p)}P_{\frak p}.$$
This is not enough to get any useful bound because the probabilities $P_{\frak p}$ decrease too slowly. We have to take care of the prime ideals with big norms separately. If we knew that for all elements $s_1\neq s_2$ of $S$ we have $N((s_1-s_2))<\kappa$, then $S$ itself would be automatically almost uniformly distributed modulo powers of every prime ideal $\frak p$ with $N(\frak p)\geq\kappa$. Otherwise, we would have that for some $s_1,s_2\in S$, $(s_1-s_2)\in\frak p$, but then $(s_1-s_2)\mathcal{O}_K\subset \frak p$, and so $N((s_1-s_2))\geq N(\frak p)$. Consequently, if $|S|\geq n+1$, then $S$ would contain $n+1$ elements which are almost uniformly distributed modulo powers of such ideals. Let $E_R$ denote the event that $S$ is contained in $B_{\|\cdot\|}(R)$, the ball of radius $R$ around the origin. Since $N(a)=O(\|a\|^d)$, we can find a constant $c_0$ such that $$\sup_{a_1,a_2\in B_{\|\cdot\|}(R)}|N((a_1-a_2))|\leq \sup_{a\in B_{\|\cdot\|}(2R)}|N(a)|\leq c_0R^d.$$ 
By the previous remark, if $|S|\geq n+1$ and the event $E_R$ happened, then $S$ contains an $(n+1)$-element subset which is almost uniformly distributed modulo every prime ideal with $N(\frak p)>c_0R^d$. Recall that by the construction, the set $S$ must have at least $n+1$ distinct elements modulo $L!$, so clearly $|S|\geq n+1$. We get the following estimate
$$P\leq \sum _{L<N(\frak p)<c_0R^d}P_{\frak p}+(1-\Pb{E_R}).$$
The last step is to relate $(1-\Pb{E_R})$ to $\mu(\{a\in \mathcal{O}_K|\|a\|>\kappa\})$. Obviously $(1-\Pb{E_R})\leq\sum_{i=1}^{n+m}\Pb{\|\xi_i\|>R}$. Recall that $\xi_i=a_i+L!\psi_i$, where $\psi_i$ has  distribution $\mu$, so $$\Pb{\|\xi_i\|>R}\leq\Pb{\|\psi_i\|>\frac{R-\|a_i\|}{L!}}.$$ The set $\{a_1,\ldots,a_{n+m}\}$ has to be chosen so that the first $n+1$ elements are almost uniformly distributed modulo powers of $\frak p$ for every prime ideal $\frak p$ of norm not exceeding $L$. We also want it to have $n+1$ distinct elements modulo $L!$. Such a set may  always be found in the ball of radius $c_1=c_1(L)$. Hence there exists a constant $c_2=c_2(L)$ such that $\Pb{\|\xi_i\|>R}\leq \Pb{\|\psi_i\|>R/{c_2}}=\mu(\{a\in \mathcal{O}_K|\|a\|>R/{c_2}\})$ for large $R$. Substituting $\kappa=R/{c_2}$ and $c=c_0c_2^d$, we get 
$$P\leq \sum _{L<N(\frak p)<c\kappa^d}P_{\frak p}+(n+m)\mu(\{a\in \mathcal{O}_K|\|a\|>\kappa\}).$$
By Proposition \ref{p.ProbEst2}, 
$$P_{\frak p}\leq C_0\left(\frac{\sum_{\chi\in\widehat{\mathcal{O}_K/\frak p}}|\widehat{\mu_\frak p}(\chi)|}{|\mathcal{O}_K/\frak p|}\right)^m$$ 
for some constant $C_0$ depending only on $n$ and $m$, so finally we get
$$P\leq C_0\sum _{L<N(\frak p)<c\kappa^d}\left(\frac{\sum_{\chi\in\widehat{\mathcal{O}_K/\frak p}}|\widehat{\mu_\frak p}(\chi)|}{|\mathcal{O}_K/\frak p|}\right)^m+(n+m)\mu(\{a\in \mathcal{O}_K|\|a\|>\kappa\}).\qedhere$$\end{proof}
\subsection{Construction of the measure $\mu$}
Let $\mathcal{O}_K= \Z\omega_1\oplus\ldots\oplus\Z\omega_d\simeq \Z^d$ for some $\omega_1,\ldots,\omega_d$. Let $M$ be a positive integer (we will specify later how big it should be). Let $\nu_i=\frac{1}{2}(\delta_{\omega_i}+\delta_{-\omega_i})$ where $\delta_{\omega}$ denotes the Dirac measure centered at $\omega$ and set $\nu=\nu_1 *\cdots * \nu_d$. Consider a random walk $X^{(i)}$ on $\mathcal{O}_K$ defined by: \begin{itemize}
\item $X^{(0)}=0$,
\item $\Pb{X^{(i+1)}=a\mid X^{(i)}=b}=\nu(a-b).$
\end{itemize}
Let $\mu$ be equal to the distribution of $X^{(M)}$, i.e., $\mu=\mu_M=\nu^{*M}$. 
We construct the random set $S$ as it was described in the outline. We fix an integer $L>\max(n+d,2n+2)$. We choose elements $a_1,a_2,\ldots,a_{n+d}$ such that at least $n+1$ of them are distinct modulo $L!$ and  such that $a_1,a_2,\ldots,a_{n+1}$ are almost uniformly distributed modulo powers of every prime ideal $\frak p$ with $N(\frak p)\leq L$. Next, we take $n+d$ independent random elements   $\psi_1,\psi_2,\ldots,\psi_{n+d}$ of $\mathcal{O}_K$ with distributions equal to $\mu$ and set $S=\{\xi_1,\ldots,\xi_{n+d}\}$ with $\xi_i=a_i+L!\psi_i$. As before, $\mu_{\frak p}$ shall denote the projection of the measure $\mu$ onto $\mathcal{O}_K/\mathfrak{p}$.
\begin{lemma}\label{l.ConvEstimate}
For any prime ideal $\frak p$ in $\mathcal{O}_K$ and  an integer $M$, we have
$$\frac{\sum_{\chi\in\widehat{\mathcal{O}_K/\frak p}}|\widehat{\mu_{\frak p}}(\chi)|}{N(\frak{p})}\leq C_1\left(\frac{1}{N(\frak p)}+M^{-1/2}\right),$$
with a constant $C_1$ depending only on $d$.
\end{lemma}
\begin{proof}
The quotient $\mathcal{O}_K/\frak p$ is isomorphic to the finite field $\F_{p^{d'}}$ for some $d'\leq d$. Since $\omega_1,\omega_2,\ldots,\omega_d$ generate $\mathcal{O}_K$ as a $\Z$-module, they also generate $\mathcal{O}_K/\frak p$ over $\F_p$. After a possible change of enumeration,  we may assume that $\mathcal{O}_K/\frak p=\F_p\omega_1\oplus\ldots\oplus\F_p\omega_{d'}$, this establishes an isomorphism $\Psi:\mathcal{O}_K/\frak{p} \simeq \F_{p}^{d'}$. We will need the following claim.

\textbf{Claim:} Let $G=G_1\times\ldots\times G_l$ be a finite abelian group and let $m_1,\ldots,m_l$ be probability measures on $G_1,\ldots,G_l$, respectively. Let $m=m_1\times\ldots\times m_l$ be the product measure on $G$. Then 
$$\sum_{\chi\in\widehat G}|\widehat{m}(\chi)|=\prod_{i=1}^l\left(\sum_{\chi_i\in\widehat{G_i}}|\widehat{m_i}(\chi_i)|\right).$$
\textbf{Proof of the claim:} This follows immediately from the fact that the decomposition $G=\prod_{i=1}^l G_i$ induces an isomorphism $\prod_{i=1}^l\widehat{G_i}\simeq\widehat{G}$ given by $$(\chi_1,\ldots,\chi_l)\mapsto \big((a_1,\ldots,a_l)\mapsto\chi_1(a_1)\ldots\chi_l(a_l)\big).$$

Convolution of measures commutes with the transport of measures by a homomorphism, so $\mu_{\frak p}=\nu_{1,\frak p}^{*M}*\ldots *\nu_{d,\frak p}^{*M}$. Let $\mu_{1, \frak p}=\nu_{1,\frak p}^{*M}*\ldots *\nu_{d',\frak p}^{*M}$ and $\mu_{2, \frak p}=\nu_{d'+1,\frak p}^{*M}*\ldots *\nu_{d,\frak p}^{*M}$. By the properties of the Fourier transform we have $\widehat{\mu_{\frak p}}=\widehat{\mu_{1,\frak p}}\widehat{\mu_{2,\frak p}}$, so $|\widehat{\mu_{\frak p}}|\leq |\widehat{\mu_{1,\frak p}}|$. The last inequality follows from the fact that the Fourier transform of a probability measure is bounded by $1$ in absolute value. We get
$$\frac{\sum_{\chi\in\widehat{\mathcal{O}_K/\frak p}}|\widehat{\mu_{\frak p}}(\chi)|}{N(\frak{p})}\leq\frac{\sum_{\chi\in\widehat{\mathcal{O}_K/\frak p}}|\widehat{\mu_{1,\frak p}}(\chi)|}{N(\frak{p})}.$$
The measure $\mu_{1,\frak p}$ corresponds via the isomorphism $\mathcal{O}_K/\frak p\simeq \F_p^{d'}$ to the measure  $\nu^{*M}\times\ldots\times\nu^{*M}$, where $\nu$ is a measure on $\F_p$ given by $\nu=\frac{1}{2}(\delta_1+\delta_{-1})$. Consequently, by the claim and the previous inequality, we get
\begin{equation}\label{e.MuEst1}
\frac{\sum_{\chi\in\widehat{\mathcal{O}_K/\frak p}}|\widehat{\mu_{\frak p}}(\chi)|}{N(\frak{p})}\leq\left(\frac{1}{p}\sum_{\chi\in\widehat{\F_p}}|\widehat{\nu}(\chi)|^M\right)^{d'}.
\end{equation}
We shall estimate the expression $\frac{1}{p}\sum_{\chi\in\widehat{\F_p}}|\widehat{\nu}(\chi)|^M$. The characters of $\F_p$ are of the form $\chi_a(t)=\exp(2\pi i a t/p)$ for $a=0,1,\ldots,p-1$, so
\begin{align*}\frac{1}{p}\sum_{\chi\in\widehat{\F_p}}|\widehat{\nu}(\chi)|^M &=
\frac{1}{p}\sum_{a=0}^{p-1}|\widehat{\nu}(\chi_a)|^M =\frac{1}{p}\sum_{a=0}^{p-1}\left(\frac{\exp(2\pi i a/p)+\exp(-2\pi i a/p)}{2}\right)^M\\
&=\frac{1}{p}\sum_{a=0}^{p-1}|\cos(2\pi a/p)|^M\leq \frac{2}{p}+\int_0^1|\cos(2\pi t)|^Mdt\\
&=\frac{2}{p}+\frac{2\Gamma({\frac{M+1}{2}})}{\sqrt{\pi}M\Gamma({\frac{M}{2}})}\leq C_1'\left(\frac{1}{p}+M^{-1/2}\right)\\
\end{align*}
for some constant $C_1'$, since $\Gamma({\frac{M+1}{2}})/\sqrt{M}\Gamma({\frac{M}{2}})$ is bounded as $M\rightarrow \infty$. Now using inequality (\ref{e.MuEst1}) and the inequality between power means, we obtain
$$\frac{\sum_{\chi\in\widehat{\mathcal{O}_K/\frak p}}|\widehat{\mu_{\frak p}}(\chi)|}{N(\frak{p})}\leq C_1\left(\frac{1}{p^{d'}}+M^{-d'/2}\right)\leq C_1\left(\frac{1}{N(\frak p)}+M^{-1/2}\right),$$ where $C_1=2^{d'-1}C_1'^{d'}$.
\end{proof}
Applying again the inequality between the means, we immediately obtain the following corollary.
\begin{corollary}\label{c.Est1} We have
$$\left(\frac{\sum_{\chi\in\widehat{\mathcal{O}_K/\frak p}}|\widehat{\mu_{\frak p}}(\chi)|}{N(\frak{p})}\right)^d\leq C_2
\left(\frac{1}{N(\frak p)^d}+M^{-d/2}\right),$$%popr
where $C_2=2^{d-1}C_1^d$ is a constant depending only on $d$.
\end{corollary}
The last ingredient that we need is the following tail estimate.
\begin{lemma}\label{l.TailEst} We have
$$\lim_{M\to\infty}\mu(\{a\in \mathcal{O}_K|\|a\|>M^{1/2}(\log M)^{1/2d}\})=0.$$
\end{lemma}
\begin{proof}
Let $E_i$ for $i=1,\ldots, d$ denote the set 
$$E_i=\{a\in \mathcal{O}_K \mid a=a_1\omega_1+\ldots+a_d\omega_d, |a_i|>M^{1/2}(\log M)^{1/2d}\}.$$ Since 
$\{a\in \mathcal{O}_K \mid \|a\|>M^{1/2}(\log M)^{1/2d}\}=E_1\cup\ldots\cup E_d$, we just need to show that $\lim_{M\to\infty}\mu(E_i)=0$ for every $i$. To do this, observe that the measure of $E_i$ depends only on the projection of $\mu$ onto the $i$-th coordinate which is equal to $\nu_i^{*M}$. Hence, 
$$\mu(E_i)=\nu^{*M}(\{z\in \Z \mid |z|>M^{1/2}(\log M)^{1/2d})\},$$
where $\mu=\frac{1}{2}(\delta_1+\delta_{-1})$. By the Central Limit Theorem, $\nu^{*M}$ scaled down by a factor of $M^{1/2}$ converges to the normal distribution $\mathcal{N}(0,1)$. Consequently, $\mu(E_i)$ must converge to zero as $(\log M)^{1/2d}$ tends to infinity.
\end{proof}
Now we are ready to prove that our construction  results in an $n$-universal set provided that $L$ and $M$ are big enough. 
\begin{theorem}
Let $S$ be a random subset of $\mathcal{O}_K$ defined as before and let $P=P_{L,M}$ denote the probability that $S$ is not $n$-universal. Then $\displaystyle\lim_{L\to\infty}\displaystyle\lim_{M\to\infty}P_{L,M}=0$.
\end{theorem}
\begin{proof}
Fix $\varepsilon>0$. We need to show that for $L$ big enough and $M$ tending to infinity $P_{L,M}$ is eventually smaller that $\varepsilon$. The sum $\sum_{\frak p\in \Spec \,\mathcal{O}_K}1/N(\frak p)^d$ is well known to be convergent for $d\geq 2$. Fix $L$ large enough so that $$\sum_{N(\frak p)>L}\frac{1}{N(\frak p)^d} <\frac{\varepsilon}{2C_0C_2},$$  where $C_0$ is the constant from Proposition \ref{p.ProbCrit} and $C_2$ is the constant from Corollary \ref{c.Est1}. Let $$D_M=\{a\in \mathcal{O}_K \mid \|a\|>M^{1/2}(\log M)^{1/2d}\}.$$ By Proposition \ref{p.ProbCrit}, there exists a constant $c=c(L)>1$ such that for large $M$ we have 
$$P\leq C_0 \left(\sum_{L<N(\frak p)<cM^{d/2}(\log M)^{1/2}}\left(\frac{\sum_{\chi\in\widehat{\mathcal{O}_K/\frak p}}|\widehat{\mu_{\frak p}}(\chi)|}{N(\frak p)}\right)^d\right) + (n+d)\mu(D_M).$$
Using Corollary \ref{c.Est1}, we get
$$P\leq C_0C_2\left( \sum_{L<N(\frak p)<cM^{d/2}(\log M)^{1/2}}\left(\frac{1}{N(\frak p)^d}+M^{-d/2}\right) \right)+(n+d)\mu(D_M).$$
Let us focus on the sum on the right hand side. We have
$$ \sum_{L<N(\frak p)<cM^{d/2}(\log M)^{1/2}}\left(\frac{1}{N(\frak p)^d}+M^{-d/2}\right)\leq
\sum_{N(\frak p)>L}\frac{1}{N(\frak p)^d}+\pi_{\mathcal{O}_K}(cM^{d/2}(\log M)^{1/2})M^{-d/2},$$
where $\pi_{\mathcal{O}_K}(x)=|\{\frak p\in\Spec\, \mathcal{O}_K \mid N(\frak p)\leq x\}|$. Just like in the ordinary prime counting function, $\pi_{\mathcal{O}_K}(x)$ is asymptotically equal to $x/\log x$ (by Landau Prime Ideal Theorem, \cite[p. 267]{HRV07}). Hence, for $M$ large enough we have
$$\pi_{\mathcal{O}_K}(cM^{d/2}(\log M)^{1/2})M^{-d/2}\leq 2\frac{2cM^{d/2}(\log M)^{1/2}}{d\log M}M^{-d/2}=\frac{4c}{d(\log M)^{1/2}}.$$
The sum $\sum_{N(\frak p)>L}1/N(\frak p)^d$ is less than $\varepsilon/2C_0C_2$, so we get 
$$P\leq \frac{\varepsilon}{2}+\frac{4cC_0C_2}{d(\log M)^{1/2}}+(n+d)\mu(D_M).$$ The constants $c$, $C_0$ and $C_2$ are independent of $M$ and by Lemma \ref{l.TailEst} the last summand tends to $0$ as $M$ tends to infinity. Thus, for $M$ large enough we have $P=P_{L,M}< \varepsilon$.
\end{proof}
Let us rephrase what we have just shown. In the ring of integers $\mathcal{O}_K$ of the field $K$ we choose $n+d$ elements $\{a_1,\ldots, a_{n+d}\}$  such that the first $n+1$ are almost uniformly distributed modulo all prime powers $\mathfrak{p}^k$ with $N(\mathfrak{p})\leq L$. Next, we pick a symmetric measure $\nu=\nu_1*\ldots *\nu_d$ on $\mathcal{O}_K$ which is supported on a set generating $\mathcal{O}_K$. If we consider a random walk $X_i$ on $\mathcal{O}_K$ defined by $X_i^{(0)}=a_i$ and $\Pb{X_i^{(n+1)}=x\mid X_i^{(n)}=y}=\nu((x-y)/L!)$, then the distribution of $X_i^{(M)}$ is nothing else than the distribution of $\xi_i$ for the  same $\xi_i$ as in the proof. Therefore we have shown that $S=S_M=\{X_1^{(M)},\ldots, X_{n+d}^{(M)}\}$ is $n$-universal with high probability as $M\to\infty$, provided that $L$ is chosen big enough.

\subsection*{Acknowledgements}

We would like to express our gratitude to Paul-Jean Cahen for many valuable comments and references to the literature, to Anne-Marie Aubert for careful reading of the first version of this article, and to the anonymous referee for a number of helpful comments.

The first author gratefully acknowledges the support of the grant of the National Science Centre, Poland (NCN) no. DEC-2012/07/E/ST1/00185.

The second author was supported by a public grant as part of the
Investissement d'avenir project, reference ANR-11-LABX-0056-LMH,
LabEx LMH.

The third author would like to thank Bo\.zena Chorzemska-Szumowicz for her help with the figures.


\begin{thebibliography}{plain}
\bibitem[AC]{AC} D.~Adam, P.-J.~Cahen, \emph{Newtonian and Schinzel quadratic fields}, Journal of Pure and Applied Algebra {\bf 215} (2011), pp. 1902--1918.
\bibitem[Ami]{A} Y.~Amice, \emph{Interpolation p-adique}, Bull. Soc. Math. France {\bf 92} (1964), pp. 117–-180.
 \bibitem[Bec]{HLP52} W.~Beckner, \emph{Inequalities in Fourier Analysis}, Annals of Mathematics
Second Series, Vol. {\bf 102}, No. 1 (1975), pp. 159--182.
 \bibitem[Bha1]{B1} M.~Bhargava, \emph{P-orderings and polynomial functions on arbitrary subsets of Dedekind
 rings}, J. Reine Angew. Math. {\bf 490} (1997), pp. 101--127.
 \bibitem[Bha2]{B2} M.~Bhargava, \emph{The factorial function and generalizations}, Amer. Math. Monthly {\bf 107} (2000), pp. 783--799.
 \bibitem[Cah]{C} P.-J.~Cahen, \emph{Newtonian and Schinzel sequences in a domain}, Journal of Pure and Applied Algebra {\bf 213} (2009), pp. 2117--2133.
 \bibitem[CC]{CC} P.-J.~Cahen, J.-L.~Chabert, \emph{Integer-valued polynomials}, Mathematical Surveys and Monographs {\bf 48}, American Mathematical Society, Providence,  1997.
 %\bibitem[Con]{KC1} K.~Conrad, \emph{Characters of Finite Abelian Groups}, expository paper, {\tt http://www.math.uconn.edu/~kconrad/blurbs/grouptheory/charthy.pdf}.
\bibitem[Eis]{Eis} D.~Eisenbud, \emph{Commutative Algebra with a View Toward Algebraic Geometry},  Springer-Verlag, New York, 2000.
 \bibitem[Iha]{Ihara1} Y.~Ihara, \emph{On the Euler-Kronecker constants of global fields and primes with small norms}, Algebraic geometry and number theory \textbf{253} (2006), pp. 407--451.
\bibitem[Ire]{Ire} K. Ireland, M. Rosen, \emph{A classical introduction to 
Modern Number Theory}, Springer-Verlag, New York, 1982.
 \bibitem[Lam]{MLam14} M.~Lamoureux, \emph{Stirling's Formula in Number Fields}, Doctoral Dissertations. Paper {\bf 412} (2014), {\tt http://digitalcommons.uconn.edu/dissertations/412}.
\bibitem[Lan]{Lan}S. Lang, \emph{Algebraic number theory}, 2nd ed., Graduate Texts in Mathematics {\bf 110}, Springer-Verlag, New York, 1994.
 \bibitem[MV]{HRV07} H.L.~Montgomery, R.C.~Vaughan, \emph{Multiplicative number theory I. Classical theory},  Cambridge tracts in advanced mathematics {\bf 97} (2007).
 \bibitem[VP]{PV} V.V.~Volkov, F.V.~Petrov, \emph{On the interpolation of integer-valued polynomials}, Journal of Number Theory {\bf 133} (2013), pp. 4224--4232.
 \bibitem[Woo]{W} M.~Wood, \emph{P-Orderings: a metric viewpoint and the non-existence of simultaneous orderings}, J. Number Theory, {\bf 99} (2003), pp. 36--56.
\bibitem[Yer]{Y} J.~Yeramian, \emph{Anneaux de Bhargava}, Comm. in Algebra {\bf 32} (2004), pp. 3043--3069.
  \end{thebibliography}
\end{document}